\documentclass[11pt,letterpaper]{amsart}
\usepackage[utf8]{inputenc}
\usepackage{amsmath}
\usepackage{amsthm}
\usepackage{amsfonts}
\usepackage{amssymb}
\usepackage{xcolor}
\usepackage{hyperref}
\usepackage{graphicx}
\usepackage{textcomp} % For \texttildelow
\usepackage[subrefformat=parens,labelformat=parens]{subfig}
\usepackage[all]{xy}

\newtheorem{theorem}{Theorem}
\newtheorem{proposition}[theorem]{Proposition}
\newtheorem{lemma}[theorem]{Lemma}

\newtheorem{corollary}[theorem]{Corollary}

\newtheorem*{pingpongF}{Ping-Pong Lemma for Thompson's Group $\boldsymbol{F}$}
\newtheorem*{pingpongfree}{Ping-Pong Lemma for Free Products}
\newtheorem{question}[theorem]{Question}
\theoremstyle{definition}
\newtheorem{definition}[theorem]{Definition}

\newtheorem{example}[theorem]{Example}

\theoremstyle{remark}
\newtheorem{remark}[theorem]{Remark}

\newcommand{\newword}[1]{\textbf{#1}}

\newcommand{\D}{\mathbb{D}}
\newcommand{\N}{\mathbb{N}}
\newcommand{\C}{\mathbb{C}}
\newcommand{\R}{\mathbb{R}}
\newcommand{\Z}{\mathbb{Z}}
\newcommand{\E}{\mathcal{E}}
\newcommand{\M}{\mathcal{M}}
\newcommand{\Chat}{\widehat{\mathbb{C}}}
\DeclareMathOperator{\diam}{diam}

\title{Quasisymmetries of Finitely Ramified Julia Sets}

\author{James Belk}
\address{School of Mathematics and Statistics, University of Glasgow, Glasgow, Scotland}
\email{\href{mailto:jim.belk@glasgow.ac.uk}{jim.belk@glasgow.ac.uk}}

\thanks{The first author has been partially supported by EPSRC grant EP/R032866/1 during the creation of this paper, as well as the National Science Foundation under Grant No.\ DMS-1854367.}

\author{Bradley Forrest}
\address{Stockton University, 101 Vera King Farris Drive, Galloway, NJ, 08205, USA}
\email{\href{mailto:bradley.forrest@stockton.edu}{bradley.forrest@stockton.edu}}

\begin{document}

\begin{abstract}
We develop a theory of quasisymmetries for finitely ramified fractals, with applications to finitely ramified Julia sets.  We prove that certain finitely ramified fractals admit a naturally defined class of ``undistorted metrics'' that are all quasi-equivalent.  As a result, piecewise-defined homeomorphisms of such a fractal that locally preserve the cell structure are quasisymmetries. This immediately gives a solution to the quasisymmetric uniformization problem for topologically rigid fractals such as the Sierpi\'nski triangle.  We show that our theory applies to many finitely ramified Julia sets, and we prove that any connected Julia set for a hyperbolic unicritical polynomial has infinitely many quasisymmetries, generalizing a result of Lyubich and Merenkov.  We also prove that the quasisymmetry group of the Julia set for the rational function $1-z^{-2}$ is infinite, and we show that the quasisymmetry groups for the Julia sets of a broad class of polynomials contain Thompson's group~$F$.  \end{abstract}

\maketitle

\section{Introduction}

Classical quasiconformal geometry is concerned with quasiconformal maps between open subsets of $\R^n$ ($n\geq 2$), but the definition of a quasiconformal map does not generalize well to arbitrary metric spaces.  In~1980, Pekka Tukia and Jussi V\"{a}is\"{a}l\"{a} introduced a class of homeomorphisms between arbitrary metric spaces that they called \newword{quasisymmetries}~\cite{TukVai}. These maps are closely related to quasiconformal homeomorphisms on~$\R^n$, but the definition makes sense for arbitrary metric spaces, allowing for the study quasiconformal geometry in a very general setting.

Over the last two decades, the increasing importance of quasiconformal geometry in topology and analysis has ignited significant interest in quasisymmetries of fractals~\cite{Bonk}.  Such maps arise naturally in the study of quasi-isometries of Gromov hyperbolic spaces~\cite{BS}, and give a possible approach to Cannon's conjecture for Gromov hyperbolic groups~\cite{BK}. Quasisymmetries also arise in the study of postcritically finite rational maps in complex dynamics~\cite{Meyer}.

A metric space is \newword{quasisymmetrically rigid} if its group of quasisymmetries is finite.  In 2013, Mario Bonk and Sergei Merenkov proved the surprising result that the square Sierpi\'{n}ski carpet is quasisymetrically rigid~\cite{BoMe}. In 2016 Bonk, Merenkov, and Mikhail Lyubich proved the same result for certain Sierpi\'{n}ski carpet Julia sets~\cite{BLM}, a result which has since been generalized~\cite{QYZ}. In 2019, Lyubich and Merenkov gave the first example of a multiply connected Julia set with infinitely many quasisymmetries, namely the ``basilica'' Julia set for~$z^2-1$~\cite{LyMe}. Recently Russell Lodge, Lyubich, Merenkov, and Sabyasachi Mukherjee gave examples of Apollonian gasket Julia sets with infinitely many quasisymmetries~\cite{LLMM} and Timothy Alland showed that the Feigenbaum quadratic Julia set has infinitely many quasisymmetries \cite{Alland}.  

The most important geometric difference between a Sierpi\'{n}ski carpet and Julia sets such as the basilica or Apollonian gasket is that the basilica and Apollonian gasket are \newword{finitely ramified fractals}. Roughly speaking, this means that these Julia sets have a tree of subsets called ``cells'', where each cell intersects neighboring cells at finitely many points.  A precise definition of finitely ramified fractal was introduced by Alexander Teplyaev in 2008~\cite{Tep} as a common generalization of Kigami's postcritically finite self-similar sets~\cite{Kig1} and Strichartz's fractafolds~\cite{Str}.  Many other Julia sets are finitely ramified, as are certain well-known fractals such as the Sierpi\'{n}ski triangle and the Viscek fractal.

In this paper we develop a general theory of quasisymmetries for finitely ramified fractals.  Though Teplyaev's definition is entirely topological, we show that certain finitely ramified fractals have a distinguished quasisymmetric equivalence class of metrics that we call \newword{undistorted metrics}, and we give a complete characterization of quasisymmetries with respect to such metrics.  Among other consequences, this yields a solution to the quasisymmetric uniformization problem for ``topologically rigid'' finitely ramified fractals such as the Sierpi\'{n}ski triangle. We also define a group of homeomorphisms of a finitely ramified fractal that we call \newword{piecewise cellular}, and we show that such homeomorphisms are always quasisymmetries when the metric is undistorted.

Next we apply this theory to Julia sets for hyperbolic rational maps.  First, we give sufficient conditions for the Julia set to be finitely ramified, and we prove that under these conditions the restriction of the usual metric on the Riemann sphere is undistorted.  This means that piecewise cellular homeomorphisms of such Julia sets are quasisymmetries, and we illustrate this result by proving that the Julia set for the rational map $z^{-2}-1$ has infinitely many quasisymmetries. 
Finally we specialize to hyperbolic polynomials.  We prove that any connected Julia set for a hyperbolic polynomial is finitely ramified, and we investigate which polynomial Julia sets have infinitely many quasisyemmtries.  We have two results along these lines, the first of which is the following.

\begin{theorem}\label{thm:Theorem1}If $f$ is a unicritical polynomial whose critical point is periodic, then the Julia set for $f$ has infinitely many quasisymmetries.
\end{theorem}

Here a polynomial is \newword{unicritical} if it has only one critical point, e.g.\ any quadratic polynomial. We prove Theorem~\ref{thm:Theorem1} by showing that the quasisymmetry group for a unicritical polynomial with periodic critical point contains $\Z_m*\Z_n$ for some $m,n\geq 2$. Our second result is that the quasisymmetry groups of the Julia sets for a certain family of hyperbolic polynomials contain Thompson's group~$F$.  This family includes the ``generic'' case where all critical points are simple and periodic with disjoint cycles.  In particular, we prove the following. 

\begin{theorem}\label{thm:Theorem2}
Let $f$ be a hyperbolic quadratic polynomial with connected Julia set~$J_f$.  Then the quasisymmetry group of $J_f$ contains Thompson's group~$F$.
\end{theorem}

The rest of the introduction is organized as follows. We begin in Section~\ref{subsec:OutlineTheory} with an outline of the main theory: including our definition of finitely ramified fractals (which is slightly different from Teplyaev's); a discussion of finitely ramified Julia sets; the definitions of undistorted metrics and  piecewise cellular homeomorphisms; and statements of our main theorems about quasisymmetries of finitely ramified fractals. In Section~\ref{subsec:Uniformization}, we apply this theory to the quasisymmetric uniformization of topologically rigid fractals such as the Sierpi\'nski triangle. Finally, in Section~\ref{subsec:IntroJuliaSets} we give a general procedure for constructing quasisymmetries of Julia sets of hyperbolic rational maps, we state our main theorems regarding polynomial Julia sets with infinitely many quasisymmetries, and we discuss the consequences for cubic polynomials.

\subsection{Outline of the theory}\label{subsec:OutlineTheory}

Here we outline the theory of quasisymmetries of finitely ramified fractals, including statements of our general results about quasisymmetries of such fractals.

\subsubsection*{Finitely ramified fractals}

The idea of a finitely ramified fractal first arose in a 1981 article by Gefen, Aharony, Mandelbrot, and Kirkpatrick on the physics of percolation ~\cite{GAMK}. They used the term ``finitely ramified'' to refer to fractals that are made up of many small pieces that are connected to the rest of the fractal at only finitely many points.  In~2008, Teplyaev formalized this notion by introducing finitely ramified cell structures~\cite[Definition~2.1]{Tep}. We will use the following definition, which is slightly different from Teplyaev's.

\begin{definition}
\label{def:finitelyramifiedcellstructure}
Given a compact, connected, metrizable space~$X$, a \newword{finitely ramified cell structure} on~$X$ consists of a locally finite, rooted tree of subsets of $X$ called \newword{cells}, satisfying the following conditions:
\begin{enumerate}
    \item Every cell is compact, connected, and has nonempty interior.\smallskip
    \item The root of the tree is the whole space $X$, and each cell is the union of its children.\smallskip
    \item The intersection of any two cells at the same level is a finite set.\smallskip
    \item For every descending path $X=E_0\supseteq E_1\supseteq E_2\supseteq \cdots$ in the tree, the intersection $\bigcap_{n\geq 1} E_n$ is a single point. 
\end{enumerate}
We will refer to the cells at level $n$ of the tree as \newword{$\boldsymbol{n}$-cells}.
\end{definition}

For example, Figure~\ref{fig:SierpinskiTriangleCellStructure} shows a finitely ramified cell structure on the Sierpi\'{n}ski triangle. A metrizable space $X$ together with a finitely ramified cell structure on $X$ is called a \newword{finitely ramified fractal.}
\begin{figure}[tb]
\centering
\includegraphics{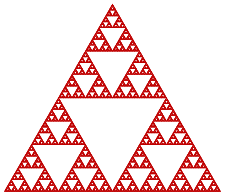}
\quad
\includegraphics{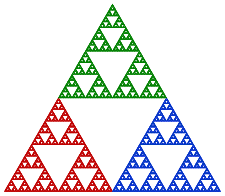}
\quad
\includegraphics{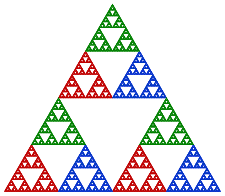}
\caption{The Sierpi\'{n}ski triangle has a natural finitely ramified cell structure with one $0$-cell, three $1$-cells, nine $2$-cells, and so forth.}
\label{fig:SierpinskiTriangleCellStructure}
\end{figure}

The main difference\footnote{There are a couple of structural differences as well.  First, we do not specify the boundary vertices of cells, since they play no role in our theory.  If explicit boundary vertices are required, it would suffice to define the boundary vertices of each cell to be the points in its topological boundary, or equivalently the set of points at which the cell intersects other cells of the same level.  Second, we have included the ``filtration'' of the cells as an explicit part of the cell structure (see~\cite[Remark~2.5]{Tep}).} between the above definition and Teplyaev's is that we do not require cells in $\mathcal{E}_n$ to be distinct from those in $\mathcal{E}_{n+1}$, since such a restriction would be inconvenient for defining cell structures on Julia sets. Despite this convention, it will be
helpful for notation to regard the collections $\mathcal{E}_n$ as disjoint. Thus our “cells”
will really be ordered pairs $(E,n)$, where $E\subseteq X$ and $n\in \mathbb{N}$ determines the level of the cell.

The notion of a finitely ramified fractal is similar to, but more general than, the notion of a postcritically finite (p.c.f.) fractal introduced by Kigami~\cite{Kig1}. See \mbox{\cite[Example~8.9]{Tep}} for an example of a fractal which is finitely ramified but postcritically infinite.

\subsubsection*{Finitely ramified Julia sets}

Recall that every rational map $f\colon \Chat\to\Chat$ has an associated \newword{Julia set}~$J_f$, which can be defined as the closure of the set of repelling periodic points of~$f$.  This set is fully invariant under~$f$, i.e.~$f^{-1}(J_f)=f(J_f)=J_f$.  We are interested in conditions under which $J_f$ has the structure of a finitely ramified fractal.
 
 If $J_f$ is connected, define a \newword{branch cut} for $J_f$ to be any closed subset $S\subseteq J_f$ for which $f$ is injective on each component of $J_f\setminus f^{-1}(S)$.  (Note then that each component of $J_f\setminus S$ is the domain for a single-valued branch of~$f^{-1}$.)  Such a branch cut is \newword{finite} if $S$ is a finite set, and \newword{invariant} if $f(S)\subseteq S$.  The following basic theorem is proven in Section~\ref{subsec:FiniteInvariantBranchCuts}.

\newcommand{\textFinitelyRamifiedJulia}{Let $f\colon \Chat\to \Chat$ be a hyperbolic rational map with connected Julia set~$J_f$.  If $J_f$ has a finite invariant branch cut~$S$, then $J_f$ admits a finitely ramified cell structure whose $n$-cells (for $n\geq 1$) are the closures of the connected components of $J_f\setminus f^{-n}(S)$.}
\begin{theorem}\label{thm:FinitelyRamifiedJulia}
\textFinitelyRamifiedJulia
\end{theorem}

We will refer to the finitely ramified cell structure described by this theorem as the \newword{induced finitely ramified cell structure} on~$J_f$.

In the statement of Theorem~\ref{thm:FinitelyRamifiedJulia}, a rational map $f$ is \newword{hyperbolic} if it is expanding on a neighborhood of $J_f$ with respect to some conformal metric (see \mbox{\cite[\S 19]{Milnor}}).  Equivalently, $f$ is hyperbolic if the forward orbit of every critical point converges to an attracting cycle \mbox{\cite[Theorem~19.1]{Milnor}}. Julia sets for hyperbolic rational maps behave much better than arbitrary Julia sets, e.g.~if $f$ is hyperbolic and $J_f$ is connected, then $J_f$ is also locally connected \mbox{\cite[Theorem~19.2]{Milnor}} and the restriction $f\colon J_f\to J_f$ is a covering map.   Hyperbolic maps are conjectured to form a dense open set in the parameter space of all rational maps \mbox{\cite[Conjecture~HD]{McMullen2}}, and in particular a quadratic polynomial $f(z)=z^2+c$ is conjectured to be hyperbolic if and only if $c$ does not lie on the boundary of the Mandelbrot set \mbox{\cite[Conjecture~HD2$'$]{McMullen2}}.

\begin{figure}
    \centering
    \includegraphics{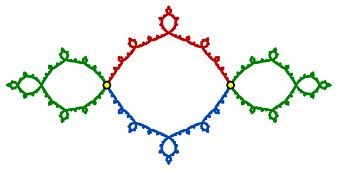}\quad \includegraphics{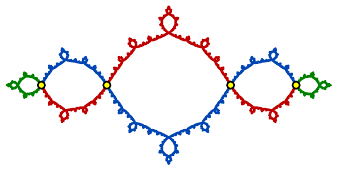} \\[-16pt]
    \includegraphics{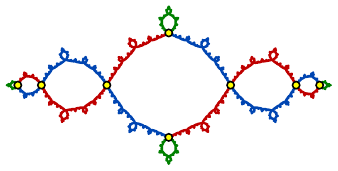}
    \caption{A finitely ramified cell structure for the basilica Julia set with four $1$-cells, eight $2$-cells, and sixteen $3$-cells.}
    \label{fig:BasilicaCellStructure}
\end{figure}
\begin{example}[The basilica]
\label{ex:Basilica}The Julia set for the polynomial $f(z)=z^2-1$ is known as the \newword{basilica}.  This polynomial has two fixed points, one of which is $p=\bigl(1-\sqrt{5}\hspace{0.08333em}\bigr)\bigr/2 \approx -0.618$. The one-point set $S=\{p\}$ is a finite invariant branch cut for~$J_f$.  By Theorem~\ref{thm:FinitelyRamifiedJulia}, there is a corresponding finitely ramified cell structure on $J_f$ whose $n$-cells are the closures of the components of~$J_f\setminus f^{-n}(p)$.  The $1$-cells, $2$-cells, and $3$-cells for this cell structure are shown in Figure~\ref{fig:BasilicaCellStructure}. 
\end{example}

It is not hard to show that a connected Julia set for a hyperbolic polynomial always admits a finite invariant branch cut (see Proposition~\ref{prop:PolynomialJuliaSetsFinitelyRamified}), and hence connected Julia sets for hyperbolic polynomials are always finitely ramified. 
 However, Theorem~\ref{thm:FinitelyRamifiedJulia} also defines finitely ramified cell structures for the Julia sets of certain non-polynomial rational maps.

\begin{figure}
    \centering
    \includegraphics{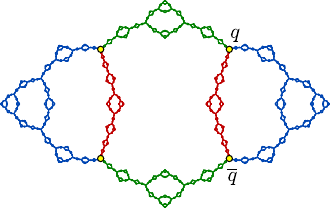}\quad \includegraphics{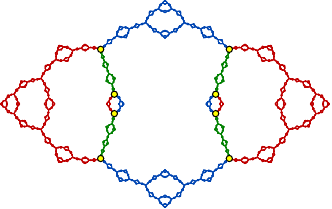}
    \caption{A finitely ramified cell structure for the bubble bath Julia set with six $1$-cells and twelve $2$-cells.}
    \label{fig:BubbleBathCellStructure}
\end{figure}
\begin{example}[The bubble bath]\label{ex:BubbleBath}The rational map $f(z)= 1-z^{-2}$ is hyperbolic and has a connected Julia set~$J_f$, which we refer to as the \newword{bubble bath} (see~\cite{Wei}).  This map has a real fixed point $p\approx -0.7549$ and two complex conjugate fixed points
 \[
 q\approx 0.8774+0.7449i\qquad\text{and}\qquad \overline{q}\approx 0.8774-0.7449i,
 \]
 and we prove in Section~\ref{subsec:bubblebath} that the set $S=\{q,\overline{q}\}$ is a branch cut for~$J_f$. By Theorem~\ref{thm:FinitelyRamifiedJulia}, there is a corresponding finitely ramified cell structure on $J_f$ whose $n$-cells are the closures of the components of~$J_f\setminus f^{-n}(p)$.  The $1$-cells and  $2$-cells for this cell structure are shown in Figure~\ref{fig:BubbleBathCellStructure}.
 \end{example}

\begin{remark}
If $S$ is a finite invariant branch cut for $J_f$, then the closures of the components of $J_f\setminus S$ (i.e.~the $1$-cells) form a Markov partition of~$J_f$ in the sense of Ruelle~\cite{Ruelle}. From this point of view, Theorem~\ref{thm:FinitelyRamifiedJulia} is a version of the familiar fact that a Markov partition for an expanding map has a well-defined coding by a subshift of finite type. In the case of quadratic polynomials, a similar strategy is used to define the well-known Yoccoz puzzle~\cite{HubbardYoccoz}.
\end{remark}

\subsubsection*{Quasisymmetries}
A homeomorphism $f\colon X\to Y$ between metric spaces is called a \newword{quasisymmetry} if there exists a homeomorphism $\eta\colon [0,\infty)\to [0,\infty)$ such that
\[
\frac{d_{\scriptscriptstyle Y}\bigl(f(a),f(b)\bigr)}{d_{\scriptscriptstyle Y}\bigl(f(a),f(c)\bigr)} \leq \eta\biggl(\frac{d_{\scriptscriptstyle X}(a,b)}{d_{\scriptscriptstyle X}(a,c)}\biggr)
\]
for every triple $a,b,c$ of distinct points in~$X$. A homeomorphism $f$ which is quasisymmetric with respect to a given  $\eta$ is called an \newword{$\boldsymbol{\eta}$-quasisymmetry.}

The quasisymmetry condition given above first appeared in the work of Beurling and Ahlfors~\cite{BeAh}, who proved that the quasisymmetries of a circle are precisely the restrictions of the quasiconformal homeomorphisms of a closed disk to its boundary.  The term ``quasisymmetry'' was later coined by Kelingos~\cite{Kel}, and the modern definition of a quasisymmetric homeomorphism between metric spaces was given by Tukia and V\"{a}is\"{a}l\"{a} in 1980~\cite{TukVai}.  V\"{a}is\"{a}l\"{a} established a strong general relationship between quasiconformal homeomorphisms and quasisymmetries with his ``egg yolk principle''~\cite{Vai1}, which implies that every quasiconformal homeomorphism $h\colon U\to V$ between domains in~$\mathbb{R}^n$ restricts to a quasisymmetry on each compact subset of~$U$.   See \cite{Hei} for a general introduction to quasisymmetries.

Compositions and inverses of quasisymmetries are again quasisymmetries, and therefore the self-quasisymmetries of any metric space $X$ form a group $QS(X)$.  Two metric spaces $X$ and $Y$ are said to be \newword{quasisymmetrically equivalent} if there exists a quasisymmetry $X\to Y$, in which case $QS(X)\cong QS(Y)$.  Similarly, two continuous metrics $d$ and $d'$ on a topological space $X$ are \newword{quasisymmetrically equivalent} if the identity map $(X,d)\to (X,d')$ is a quasisymmetry, in which case~$QS(X,d) = QS(X,d')$.

\subsubsection*{Undistorted metrics}
Our main general result about finitely ramified fractals is to prove that certain finitely ramified fractals admit a natural family of quasisymmetrically equivalent metrics. 

\begin{definition}
\label{def:undistorted}
If $X$ is a finitely ramified fractal, we say that a metric $d$ on $X$ is \newword{undistorted} if it satisfies the following conditions:
\begin{enumerate}
    \item \textbf{Exponential Decay:} There exist constants $0< r \leq R < 1$ and $C\geq 1$ so that
    \[
    \frac{r^{n-m}}{C} \leq \frac{\diam(E')}{\diam(E)} \leq C R^{n-m}
    \]
    for every $m$-cell $E$ and $n$-cell $E'$ with $n\geq m$ such that $E\cap E'\ne\emptyset$. \smallskip
    \item \textbf{Cell Separation:} There exists a constant $\delta>0$ so that
    \[
    d(E_1,E_2) \geq \delta \diam(E_1)
    \]
    for every pair $E_1,E_2$ of disjoint $n$-cells.
\end{enumerate}
\end{definition}

If we wish to specify the constants $r$, $R$, $C$, and $\delta$, then we will say that the metric $d$ is \newword{$\boldsymbol{(r,R,C,\delta)}$-undistorted}.  For example, it is easy to check that the restriction of the Euclidean metric to the Sierpi\'nski triangle is $\bigl(1/2, 1/2, 1, \sqrt{3}/2\bigr)$-undistorted with respect to the cell structure shown in Figure~\ref{fig:SierpinskiTriangleCellStructure}.

Note that the cell separation condition could be written
\[
d(E_1,E_2)\geq \delta \max\{\diam(E_1),\diam(E_2)\}.
\]
This condition is very similar to the ``uniform relative separation'' condition  introduced by Bonk for the uniformization of plane Sierpi\'nski carpets~\cite{Bonk2011}, though we compare only cells of the same level and we use a max instead of a min.

\newcommand{\textCompaibleMetric}{Let $X$ be a finitely ramified fractal. Then:
\begin{enumerate}
\item Any two undistorted metrics on $X$ are quasisymmetrically equivalent.\smallskip
\item If a metric $d$ on $X$ is undistorted, then so is any metric which is quasisymmetrically equivalent to~$d$.
\end{enumerate}
}%
\begin{theorem}\label{thm:CompatibleMetric}\textCompaibleMetric
\end{theorem}

We prove this theorem in Section~\ref{sec:CompatibleMetricProof}.  Not every finitely ramified fractal $X$ admits an undistorted metric (see Theorem~\ref{thm:AdmitsUndistortedMetric} below), but if such a metric exists then Theorem~\ref{thm:CompatibleMetric} tells us that it is unique up to quasisymmetric equivalence.   In particular, the quasisymmetry group $QS(X)$ with respect to an undistorted metric does not depend on the choice of undistorted metric, but instead depends only on the topology of $X$ together with the cell structure.  It would be interesting to find an explicit topological characterization of quasisymmetries in this case.

The following example illustrates undistorted metrics and Theorem~\ref{thm:CompatibleMetric}.

\begin{figure}
\centering
$\raisebox{-0.48\height}{\includegraphics{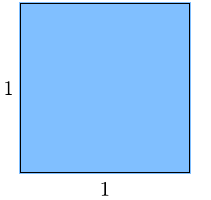}}\hfill%
\raisebox{-0.48\height}{\includegraphics{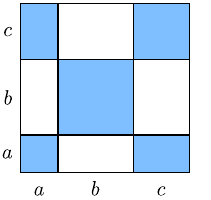}}\hfill%
\raisebox{-0.48\height}{\includegraphics{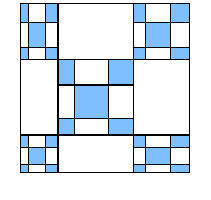}}$
\caption{The first three steps in the construction of the self-affine fractal $V(a,b,c)$.}
\label{fig:VicsekSteps}
\end{figure}%
\begin{figure}
\centering
$\underset{\textstyle V(1/3,1/3,1/3)\rule{0pt}{12pt}}{\includegraphics{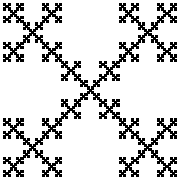}}\qquad\qquad%
\underset{\textstyle V(1/4,1/2,1/4)\rule{0pt}{12pt}}{\includegraphics{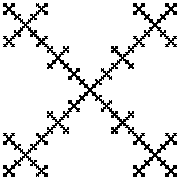}}\qquad\qquad%
\underset{\textstyle V(1/2,1/4,1/4)\rule{0pt}{12pt}}{\includegraphics{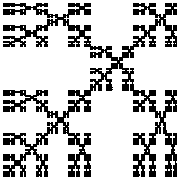}}$
\caption{Three fractals in the $V(a,b,c)$ family, the first of which is the famous Vicsek fractal. The Euclidean metric is undistorted for the first two, but distorted for the third.}
\label{fig:VicsekExamples}
\end{figure}%
\begin{example}[The Vicsek Family]\label{ex:Vicsek}Given real numbers $a,b,c\in(0,1)$ with $a+b+c=1$, let $V(a,b,c)$ be the self-affine fractal constructed using the procedure shown in Figure~\ref{fig:VicsekSteps}, where we iteratively replace each rectangle with five subrectangles whose side lengths have the given ratios.  Figure~\ref{fig:VicsekExamples} shows three examples of the resulting fractal $V(a,b,c)$ for different values of the parameters, with the $(a,b,c)=(1/3,1/3,1/3)$ case being the famous Vicsek fractal.

Each of the fractals $V(a,b,c)$ has a natural finitely ramified cell structure with exactly $5^n$ different $n$-cells for each $n\geq 0$.  In the case where $a=c$, it is easy to check that the Euclidean metric restricts to an undistorted metric on $V(a,b,c)$, so all such fractals are quasisymmetrically equivalent, even though the obvious homeomorphisms between these fractals are not bilipschitz. For $a\ne c$, the restriction of the Euclidean metric to $V(a,b,c)$ is distorted (i.e.~not undistorted), since for every $n\geq 0$ there exist a pair of intersecting $n$-cells with diameters $\sqrt{2}\,ac^{n-1}$ and $\sqrt{2}\,ba^{n-1}$, respectively. 
\end{example}

The following theorem (proven in Section~\ref{subsec:ProofThm1-2}) characterizes the finitely ramified fractals that admit undistorted metrics.

\newcommand{\textAdmitsUndistortedMetric}{Let $X$ be a finitely ramified fractal.  Then $X$ admits an undistorted metric if and only if there exists a constant $k\in\N$ that satisfies the following conditions:
\begin{enumerate}
    \item Every $n$-cell in $X$ contains at least two disjoint $(n+k)$-cells.\smallskip
    \item No two disjoint $n$-cells in $X$ intersect a common $(n+k)$-cell.
\end{enumerate}
}%
\begin{theorem}\label{thm:AdmitsUndistortedMetric}\textAdmitsUndistortedMetric
\end{theorem}

The two conditions in Theorem \ref{thm:AdmitsUndistortedMetric} are combinatorial versions of exponential decay and cell separation respectively. Indeed, both exponential decay and cell separation are uniform conditions and Theorem \ref{thm:AdmitsUndistortedMetric} also enforces this uniformity.  While we will focus our attention on finitely ramified fractals with undistorted metrics, Example \ref{ex:NoUndistorted} presents a finitely ramified fractal that does not admit an undistorted metric.

It is easy to check that the limit spaces for expanding replacement systems defined in \cite{BeFo2} satisfy the two conditions above, and therefore admit undistorted metrics.

\subsubsection*{Piecewise cellular homeomorphisms}
As a consequence of Theorem~\ref{thm:CompatibleMetric}, we get the following useful test for whether a given homeomorphism is a quasisymmetry.

\begin{corollary}\label{cor:PullbackTest}
Let $(X,d_X)$ be a finitely ramified fractal with undistorted metric, let $(Y,d_Y)$ be a metric space, and let $f\colon X\to Y$ be a homeomorphism.  Then $f$ is a quasisymmetry if and only if the pullback metric $f^*(d_Y)$ on $X$ is undistorted.
\end{corollary}

That is, $f$ is a quasisymmetry if and only if the images in $Y$ of the cells of $X$ satisfy the exponential decay condition and the cell separation condition.

We can use this corollary to define a large class of quasisymmetries between finitely ramified fractals. First, we say that a homeomorphism between an $m$-cell and an $n$-cell is \newword{cellular} if it maps $(m+k)$-cells to $(n+k)$-cells for each $k\geq 0$. A homeomorphism $f\colon X\to Y$ between finitely ramified fractals is \newword{piecewise cellular} if we can subdivide $X$ and $Y$ into cells $E_1,\ldots,E_k$ and $E_1',\ldots,E_k'$, respectively, such that $f$ maps each $E_i$ cellularly to $E_{i}'$.  We prove the following theorem in Section~\ref{subsec:PiecewiseCellularAreQuasisymmetries}.

\newcommand{\textPiecewiseCellular}{%
Let $X$ and $Y$ be finitely ramified fractals with undistorted metrics. Then any piecewise cellular homeomorphism from $X$ to $Y$ is a quasisymmetry.%
}
\begin{theorem}\label{thm:PiecewiseCellular}\textPiecewiseCellular
\end{theorem}

For a single finitely ramified fractal~$X$ with an undistorted metric, the piecewise cellular homeomorphisms $X\to X$ form a subgroup of the full quasisymmetry group.  For example:
\begin{enumerate}
    \item If we put a finitely ramified cell structure on the interval $[0,1]$ for which each $n$-cell splits into two $(n+1)$-cells, then the group of orientation-preserving piecewise cellular homeomorphisms of $[0,1]$ is Thompson's group~$F$.\smallskip
    \item The group of piecewise cellular homeomorphisms of the Vicsek fractal (see Example~\ref{ex:Vicsek}) is uncountable, and hence the Vicsek fractal has uncountably many quasisymmetries.
\end{enumerate}
Note that the rearrangements defined in \cite{BeFo2} always act on the corresponding limit space by piecewise-cellular homeomorphisms.  It follows that any rearrangement of the limit space for an expanding replacement system is a quasisymmetry with respect to any undistorted metric.

\subsection{Application to quasisymmetric uniformization}\label{subsec:Uniformization}

The \newword{quasisymmetric uniformization problem} for a metric space~$Y$ asks how to tell whether a given metric space~$X$ which is homeomorphic to $Y$ is quasisymmetrically equivalent to~$Y$ \mbox{\cite[Section~3]{Bonk}}.  This problem has been solved for the closed interval and circle~\cite{TukVai} and the standard middle-thirds Cantor set~\cite{DaSe}.  Additionally, uniformization theorems have been obtained for round Sierpi\'nski carpets in the plane ~\cite{Bonk2011}, geodesic trees \cite{BoMey,BoMey2}, Cantor circle Julia sets~\cite{QiYa}, circle domains in the plane~\cite{MeWi}, the \mbox{$2$-sphere}~\cite{BoKl2}, and other compact orientable surfaces~\cite{Wil} \cite{GeyWil}, though quasisymmetric uniformization for higher-dimensional spheres and Euclidean spaces appears to be difficult~\cite{Semmes1,Semmes2}.

Theorem~\ref{thm:CompatibleMetric} gives a solution to the quasisymmetric uniformization problem for certain finitely ramified fractals.  Specifically, we say that a finitely ramified fractal $Y$ is \newword{topologically rigid} if every self-homeomorphism of $Y$ preserves the cell structure. Bandt and Retta have shown that many well-known fractals are topologically rigid~\cite[Theorem~4.2]{BaRe}, including the Sierpi\'{n}ski triangle, the $n$-dimensional Sierpi\'{n}ski gasket, the pentagasket and hexagasket (see~\cite{StrNotices}), and so forth. For example, the Sierpi\'{n}ski triangle is topologically rigid because the three points at which the $1$-cells intersect are the only three points whose removal disconnects the fractal into three pieces, so any homeomorphism must permute these points and hence the $1$-cells, and by induction the $n$-cells. As observed by Jasper Weinburd~\cite[Chapter~4]{Wei}, topological rigidity is also common for Julia sets of rational maps, with one such example shown in Figure~\ref{fig:BirdsJulia}.  Note that the Sierpi\'{n}ski triangle also arises as the Julia set of a rational map~\cite{Ushiki}.
\begin{figure}
\centering
\includegraphics{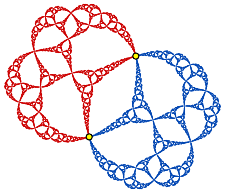}\hfill\includegraphics{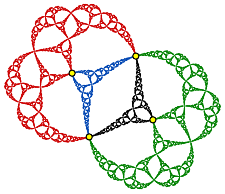}\hfill\includegraphics{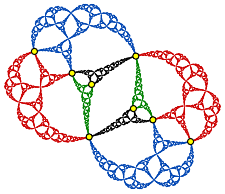}
\caption{The Julia set for $f(z) = (e^{2\pi i/3} z^2-1)/(z^2-1)$ has a finitely ramified cell structure with $2^n$ different $n$-cells for each~$n$. Every homeomorphism of the Julia set preserves this cell structure, and the homeomorphism group has order~$8$.}
\label{fig:BirdsJulia}
\end{figure}

Given a finitely ramified fractal $Y$ and a space $X$ that is homeomorphic to~$Y$, there is a finitely ramified cell structure on $X$ obtained by pulling back the cell structure on $Y$.  Furthermore, if $Y$ is topologically rigid, the cell structure on $X$ is canonical in that it is independent of the homeomorphism.  As a corollary to Theorem~\ref{thm:CompatibleMetric}, we obtain the the following solution to the quasisymmetric uniformization problem for topologically rigid finitely ramified fractals.

\begin{corollary}
\label{cor:quasisymrigid}
Let $Y$ be a finitely ramified fractal which is topologically rigid and whose metric is undistorted, and let $X$ be a metric space which is homeomorphic to~$Y$.  Then $X$ is quasisymmetrically equivalent to $Y$ if and only if the metric on $X$ is undistorted (with respect to the cell structure on $X$ induced by the homeomorphism).
\end{corollary}

Corollary \ref{cor:quasisymrigid} gives a solution to the quasisymmetric uniformization problem for the Sierpi\'{n}ski triangle and the $n$-dimensional Sierpi\'{n}ski gasket.  This solution is nontrivial in the sense that there are many metrics on these fractals that are not quasisymmetrically equivalent to the standard ones.  For example, Figure~\ref{fig:BadGaskets}(a) shows the \newword{Rauzy gasket} (see~\cite{ArnSta}), which is defined using  an iterated function system on a triangle similar to that of the Sierpi\'{n}ski triangle, except that the three maps are projective instead of affine linear.  Though it is homeomorphic to the Sierpi\'{n}ski triangle, the Rauzy gasket has the property that the maximum diameter of an \mbox{$n$-cell} is on the order of $1/n$. This violates the exponential decay condition for an undistorted metric, so it follows that the Rauzy gasket is not quasisymmetrically equivalent to the Sierpi\'{n}ski triangle.  Similarly, Figure~\ref{fig:BadGaskets}(b) shows a fractal homeomorphic to the Sierpi\'{n}ski triangle which fails the cell separation condition near the points where the circular sectors are tangent.
\begin{figure}
\centering
$\underset{\textstyle\text{(a)}}{\includegraphics{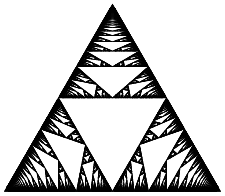}}
\qquad\underset{\textstyle\text{(b)}}{\includegraphics{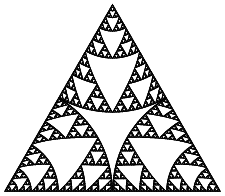}}$
\caption{(a) The Rauzy gasket (b) A Sierpi\'{n}ski triangle whose three main subtriangles have been expanded to circular sectors via radial projection. Neither of these fractals is quasisymmetrically equivalent to a standard Sierpi\'{n}ski triangle.}
\label{fig:BadGaskets}
\end{figure}

\begin{remark}Note that any self-homeomorphism of a topologically rigid finitely ramified fractal must permute the cells at each level.  It follows that every self-homeomorphism is a quasisymmetry with respect to any undistorted metric, and the full group of homeomorphisms must be residually finite. Indeed, many topologically rigid fractals have \textit{finite} homeomorphism groups. 
 For example, the homeomorphism group of the Sierpi\'nski triangle is dihedral of order~$6$ \cite{Lou}, and the homeomorphism group of the Julia set in Figure~\ref{fig:BirdsJulia} is dihedral of order~$8$ (with a faithful action on the four 2-cells).  Such fractals are trivially quasisymmetrically rigid. 
\end{remark}

\subsection{Applications to Julia sets}\label{subsec:IntroJuliaSets}
As we previously mentioned, the Julia set for any hyperbolic polynomial has a finite invariant branch cut (see Proposition~\ref{prop:PolynomialJuliaSetsFinitelyRamified}), as do the Julia sets for some hyperbolic rational maps such as the bubble bath.  By Theorem~\ref{thm:FinitelyRamifiedJulia}, such Julia sets inherit a natural finitely ramified cell structure.  The following theorem, which we prove in Section~\ref{subsec:PiecewiseCanonical}, lets us apply our theory of undistorted metrics to obtain quasisymmetries of such Julia sets.

\newcommand{\textJuliaSetUndistorted}{%
Let $f\colon \Chat\to \Chat$ be a hyperbolic rational function, and suppose that the Julia set $J_f$ is connected and has a finite invariant branch cut.  Then the restriction of the spherical metric on\/ $\Chat$ to $J_f$ is undistorted with respect to the induced finitely ramified cell structure.
}
\begin{theorem}\label{thm:JuliaSetUndistorted}\textJuliaSetUndistorted
\end{theorem}

\subsubsection*{Piecewise canonical homeomorphisms}
Theorem~\ref{thm:JuliaSetUndistorted} gives a large class of Julia sets whose metrics are undistorted. By Theorem~\ref{thm:PiecewiseCellular}, any piecewise cellular homeomorphism of such a Julia set is a quasisymmetry.  As we will now describe, such Julia sets have a very natural class of piecewise cellular homeomorphisms which are also piecewise conformal.

If $f\colon \Chat\to\Chat$ is a hyperbolic rational map and $A$ and $B$ are connected subsets of $J_f$, we say that a homeomorphism $g\colon A\to B$ is \newword{canonical} if there exist positive integers $m$ and $n$ so that $f^m\circ g$ agrees with $f^n$ on~$A$, i.e.~$g$ is a branch of $f^{-m}\circ f^n$ on~$A$.  Note that restrictions, compositions, and inverses of canonical homeomorphisms are canonical, and that canonical homeomorphisms map periodic and preperiodic points of $J_f$ to periodic and preperiodic points.

\begin{figure}
\centering
\includegraphics{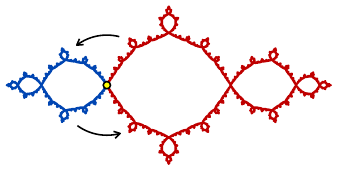}
\caption{An order-two quasisymmetry of the basilica.  This homeomorphism is piecewise canonical with a single breakpoint at $p=\bigl(1-\sqrt{5}\bigr)\bigr/2$. The blue component maps to the red component by~$f(z)=z^2-1$, and the red component maps to the blue component by a branch of~$f^{-1}$.}
\label{fig:BasilicaRearrangement}
\end{figure}
We say that a homeomorphism $h\colon J_f\to J_f$ is \newword{piecewise canonical} if there exists a finite set $B\subseteq J_f$ of \newword{breakpoints} such that $h$ restricts to a canonical homeomorphism on each component of $J_f\setminus B$.  For example, Figure~\ref{fig:BasilicaRearrangement} shows a piecewise canonical homeomorphism of the basilica. (See~\cite{BeFo1} for many more such examples.) The following theorem, which we prove in Section~\ref{subsec:PiecewiseCanonical}, shows that piecewise canonical homeomorphisms are often quasisymmetries.

\newcommand{\textPiecewiseCanonicalQuasisymmetries}{Let $f\colon \Chat\to\Chat$ be a hyperbolic rational map, and suppose the Julia set $J_f$ is connected and has a finite invariant branch cut. Then every piecewise canonical homeomorphism of $J_f$ whose breakpoints are periodic or preperiodic under~$f$ is a quasisymmetry.
}%
\begin{theorem}\label{thm:PiecewiseCanonicalQuasisymmetries}\textPiecewiseCanonicalQuasisymmetries
\end{theorem}

Note that the collection of piecewise canonical homeomorphisms of a Julia set whose breakpoints are periodic or pre-periodic forms a group.  These groups are closely related to the rearrangement groups described by the authors in~\cite{BeFo2}, though to obtain a rearrangement group, one must further restrict the allowed breakpoints to lie in finitely many grand orbits.  For example, the basilica Thompson group $T_B$ defined by the authors in \cite{BeFo1} is precisely the group of piecewise canonical homeomorphisms of the basilica whose breakpoints lie in the grand orbit of the fixed point at $p=\bigl(1-\sqrt{5}\bigr)\bigr/2$. Lyubich and Merenkov proved that the elements of $T_B$ are quasisymmetries, and that this group is in a certain sense dense in the group of planar quasisymmetries of the basilica~\cite{LyMe}.

\subsubsection*{Infinitely many quasisymmetries}
Theorem~\ref{thm:PiecewiseCanonicalQuasisymmetries} allows us to construct infinitely many quasisymmetries for many different Julia sets.  As a first application, we show in Section~\ref{subsec:bubblebath} that the bubble bath Julia set introduced in Example~\ref{ex:BubbleBath} has infinitely many quasisymmetries.

\newcommand{\textBubbleBathModularGroup}{The Julia set for $f(z)=1-z^{-2}$ has a finite invariant branch cut.  The group of quasisymmetries is infinite, and indeed it contains the modular group\/~$\Z_2*\Z_3$.
}%
\begin{theorem}\label{thm:BubbleBathModularGroup}\textBubbleBathModularGroup
\end{theorem}

The existence of a copy of $\Z_2*\Z_3$ in the group of homeomorphisms of the bubble bath was first observed by Jasper Weinburd~\cite{Wei}.

We also prove a general theorem about certain polynomial Julia sets having infinitely many quasisymmetries. For simplicity, we consider only hyperbolic polynomials which are \newword{postcritically finite}, i.e.\ all of the critical points are periodic or pre-periodic. As defined by Douady and Hubbard~\cite{OrsayNotes}, every postcritically finite polynomial has a \newword{Hubbard tree}, which is a certain finite topological tree contained in the filled Julia set (see Section~\ref{subsec:quasisymmetrygroupscontainingF}).  Our main theorem about postcritically finite polynomials is the following.

\begin{theorem}\label{thm:BigTheorem}
Let $f\colon \C\to \C$ be a postcritically finite hyperbolic polynomial of degree $d\geq 2$.
\begin{enumerate}
    \item If $f$ is unicritical, then the quasisymmetry group of $J_f$ contains the free product $\Z_d*\Z_n$ for some~$n\geq 2$.\smallskip
    \item If one of the leaves of the Hubbard tree for $f$ is contained in a periodic cycle of local degree~$2$, then the quasisymmetry group of $J_f$ contains Thompson's group~$F$.
\end{enumerate}
In particular, in both cases the quasisymmetry group is infinite.
\end{theorem}

This theorem is proven in Section~\ref{sec:JuliaSetsInfinitelyManyQuasi}. Note that part (1) of this theorem immediately implies Theorem~\ref{thm:Theorem1}, since a unicritical polynomial is postcritically finite and hyperbolic if and only if its critical point is periodic.

Theorem~\ref{thm:BigTheorem} is actually much more general than it seems.  In particular, \mbox{McMullen} proved that any connected Julia set $J_f$ for a hyperbolic rational map $f$ is quasisymmetrically equivalent to the Julia set $J_g$ for some postcritically finite hyperbolic rational map~$g$ (see~\cite[Theorem~3.4]{McMullen} or \cite[Corollary~7.36]{BraFag}).  For example, if $f(z)=z^2+c$ is a hyperbolic quadratic with connected Julia set, then $g(z) = z^2+c'$, where $c'$ is the ``center'' of the hyperbolic component of the Mandelbrot set that contains~$c$ (see~\cite{MilnorHyperbolic}). Thus Theorem~\ref{thm:BigTheorem} applies to a large class of hyperbolic polynomials.

To be precise, the quasisymmetry that McMullen defined  is known as a ``stable conjugacy'', and this preserves the mapping properties of the Fatou components.  Here the \newword{Fatou set} for a rational map $f$ is the complement of the Julia set, and the \newword{Fatou components} are the connected components of the Fatou set, which must be homeomorphic to open disks if $J_f$ is connected. If $U$ is a Fatou component then so is $f(U)$, with the degree of the mapping from $U$ to $f(U)$ determined by the total multiplicity of the critical points contained in $U$.  McMullen's stable conjugacy induces a bijection between the Fatou components of $f$ and the Fatou components of $g$, and this bijection preserves the mappings of the Fatou components as well as the mapping degrees.  Thus part (1) of the Theorem~\ref{thm:BigTheorem} applies to any hyperbolic polynomial with connected Julia set whose critical points all lie in a single Fatou component, and similarly part (2) applies to many hyperbolic polynomials that are not postcritically finite. Both parts\footnote{Note that if $f$ is a quadratic with periodic critical point then the critical value for $f$ must be a leaf of the Hubbard tree. This is because at least one point in the critical cycle must be a leaf, and the critical value must have smallest valence of all the points in the critical cycle.} apply to any hyperbolic quadratic polynomial with connected Julia set, which establishes Theorem~\ref{thm:Theorem2}.

We prove the two statements of Theorem \ref{thm:BigTheorem} independently, with statement~(1) proven in Section~\ref{subsec:Unicritical} and statement~(2) proven in Section~\ref{subsec:quasisymmetrygroupscontainingF}.  For statement (1), we observe that rotation by $2\pi/d$ {at the critical point is a quasisymmetry, producing the $\mathbb{Z}_d$ component of the free product.  Further, when the critical point is not fixed, there exists a fixed point which is the landing point of $n \geq 2$ external rays; the external rays cuts the Julia set into $n$ homeomorphic pieces, and rotation among these pieces gives the $\mathbb{Z}_n$ component of the free product.  For statement (2), we show that the action of $F$ on the boundary of the Fatou component containing the given leaf of the Hubbard tree extends to an action of $F$ on the whole Julia set by piecewise canonical homeomorphisms. 

\subsubsection*{Cubic polynomials}
Theorem~\ref{thm:BigTheorem} also has consequences for Julia sets of hyperbolic cubic polynomials. In \cite{MilnorCubic}, Milnor gives a classification of hyperbolic cubic polynomials into four types:
\begin{enumerate}
    \item[A.] (Adjacent) All critical points lie in the same Fatou component.\smallskip
    \item[B.] (Bitransitive) The critical points are distinct and lie in two different Fatou components from the same periodic cycle of Fatou components.\smallskip
    \item[C.] (Capture) The critical points are distinct and only one lies in a periodic Fatou component.\smallskip
    \item[D.] (Disjoint) The critical points are distinct and lie in two different Fatou components from two different periodic cycles of Fatou components.
\end{enumerate}
If $f$ is a hyperbolic cubic polynomial of type~A or type~D with connected Julia set, then it follows from Theorem~\ref{thm:BigTheorem} that $J_f$ has infinitely many quasisymmetries.  In particular, by McMullen's theorem we may assume $f$ is postcritically finite.  If $f$ has type~A then it is unicritical and hence has infinitely many quasisymmetries by Theorem~\ref{thm:BigTheorem}(1). If $f$ has type~D, then $f$ has two critical points that lie in two different periodic cycles, so any leaf of the corresponding Hubbard tree must lie in one of these cycles, and therefore $J_f$ has infinitely many quasisymmetries by Theorem~\ref{thm:BigTheorem}(2).

As for types B and C, there are certainly examples of Julia sets for cubic polynomials of these types that have infinitely many quasisymmetries.  For instance, the polynomial $-\tfrac34(z^3-3z)+\tfrac12$ has type~C, and the quasisymmetry group for its Julia set contains Thompson's group $F$ by Theorem~\ref{thm:BigTheorem}(2).  Similarly, the polynomial $\tfrac12 (z^3-3z)$ has type~B, but has the same second iterate, and hence the same Julia set, as its negative $-\tfrac12 (z^3-3z)$;  the negative has type~D, so this Julia set has infinitely many quasisymmetries by Theorem~\ref{thm:BigTheorem}(2).

However, there are also examples of Julia sets of cubic polynomials of types B and C for which our methods can be used to construct only finitely many quasisymmetries, as shown in the following examples.

\begin{figure}
\centering
\includegraphics{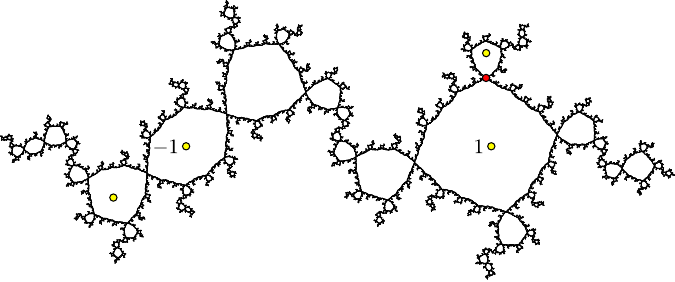}
\caption{The Julia set $J_f$ for the cubic polynomial from Example~\ref{ex:TypeBCubic}. 
The four points of the critical cycle are shown in yellow, and the period-two point $p$ is shown in red.}
\label{fig:CubicExample2}
\end{figure}
\begin{example}[A cubic of type~B]\label{ex:TypeBCubic}
Figure~\ref{fig:CubicExample2} shows the Julia set $J_f$ for a postcritically finite, hyperbolic cubic polynomial $f(z)=a(z^3-3z)+b$ of type~B, where $a\approx 0.4916+0.1527i$ and $b\approx -0.0168+0.3054i$. The two critical points $\pm 1$ for this polynomial lie in the same $4$-cycle, with $f(1)=-1$ and $f^3(-1)=1$.

This Julia set $J_f$ seems to have very few quasisymmetries. There is an order-two quasisymmetry of $J_f$ that fixes the period-two point $p\approx 0.9655+0.4479i$ (shown in red in Figure~\ref{fig:CubicExample2}) and switches the two components of $J_f\setminus\{p\}$. In the terminology of Section~\ref{subsec:Unicritical}, this is because $p$ is a rotational fixed point for $f^2$. As far as we know, this is the only nontrivial quasisymmetry of~$J_f$.  
\end{example}

\begin{figure}
\centering
\includegraphics{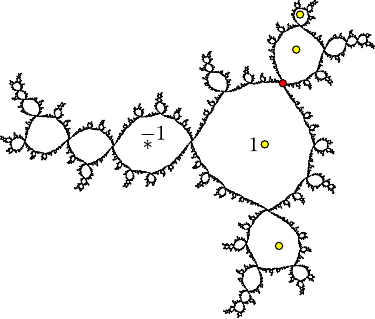}
\caption{The Julia set $J_f$ for the cubic polynomial from Example~\ref{ex:TypeCCubic}. 
The four postcritical points are shown in yellow, and the fixed point $p$ is shown in red.}
\label{fig:CubicExample}
\end{figure}
\begin{example}[A cubic of type~C]\label{ex:TypeCCubic} Figure~\ref{fig:CubicExample} shows the Julia set $J_f$ for a postcritically finite, hyperbolic cubic polynomial $f(z)=a(z^3-3z)+b$ of type~C, where $a\approx 0.0163+0.1498i$ and $b\approx 1.5704+1.9182i$. This polynomial has critical points at $1$ and $-1$, where $f^3(-1)=1$ and $1$ has period two.

Again, this Julia set seems to have very few quasisymmetries.  Specifically, there is an order-two quasisymmetry that fixes the rotational fixed point $p\approx 1.3090+1.0433i$, and as far as we know this is the only nontrivial quasisymmetry of~$J_f$.
\end{example}

\subsection{Open questions}

Our work raises several questions.  First, recall that a rational map $f$ is \newword{subhyperbolic} if the forward orbit of each critical point is either finite or converges to an attracting cycle (see~\cite[Section~19]{Milnor}).  For example, Misiurewicz points in the Mandelbrot set correspond to quadratic polynomials which are subhyperbolic but not hyperbolic, and the corresponding Julia sets are dendrites.

\begin{question}\label{question:subhyperbolic}Do the analogs of Theorems \ref{thm:JuliaSetUndistorted} and \ref{thm:PiecewiseCanonicalQuasisymmetries} hold for subhyperbolic rational maps?  That is, is the restriction of the Euclidean metric in the subhyperbolic case undistorted, and are piecewise-canonical homeomorphisms always quasisymmetries?
\end{question}

A positive answer to this question would yield the theorem that the Julia set corresponding to any Misiurewicz point in the Mandelbrot set has infinitely many quasisymmetries.  Note that a subhyperbolic map is expanding on a neighborhood of its Julia set except at finitely many cone points, so much of the proof of Theorem~\ref{thm:JuliaSetUndistorted} should go through.  It follows from the results of Matteo Tarocchi in \cite{Tarocchi} that certain dendrites admit an infinite group of quasisymmetries which is dense in the homeomorphism group, but it is unclear whether any dendrite Julia sets are quasisymmetrically equivalent to these. 

The following question highlights a particularly interesting example of a subhyperbolic map.

\begin{question}\label{qu:SierpinskiTriangle}
Let $f\colon \Chat\to\Chat$ be the map $f(z) = \bigl(z^3-\frac{16}{27}\bigr)\bigr/z$, whose Julia set $J_f$ is homeomorphic to a Sierpi\'nski triangle~\cite{Kam}.  Is $J_f$ quasisymmetrically equivalent to the standard Sierpi\'nski triangle?
\end{question}

Since the map $f$ in this question is subhyperbolic, a positive answer to Question~\ref{question:subhyperbolic} would combine with Corollary~\ref{cor:quasisymrigid} to give a positive answer to Question~\ref{qu:SierpinskiTriangle} as well. 

Though our results cover a large class of hyperbolic polynomials, the following question remains.

\begin{question}
Are there any hyperbolic polynomials $f$ with connected Julia set $J_f$ for which $J_f$ has only finitely many quasisymmetries?  In particular, does this hold for the cubic polynomials in Examples~\ref{ex:TypeBCubic} and~\ref{ex:TypeCCubic}?
\end{question}

Finally, in the case of the basilica, Neretin has shown that the basilica Thompson group defined in \cite{BeFo1} is uniformly dense in the group of orientation-preserving homeomorphisms~\cite{Neretin}.  This raises the following question.

\begin{question}
If $f$ is a hyperbolic polynomial with connected Julia set $J_f$, does the group of orientation-preserving homeomorphisms of $J_f$ always have a finitely generated, uniformly dense subgroup?
\end{question}

For example, one might hope to obtain such a subgroup by choosing some canonical branch cut and then proving that the resulting rearrangement group is dense and finitely generated (see \cite{BeFo2}), though as of yet there are few general theorems regarding the finiteness properties of rearrangement groups of Julia sets. Note that Lyubich and Merenkov have also shown that the basilica Thompson group is dense in the group of orientation-preserving quasisymmetries of the basilica in a certain quantitative sense~\cite{LyMe}, so one can similarly ask whether this result can be extended to other Julia sets.

\numberwithin{theorem}{section}

\section{Quasisymmetries of Finitely Ramified  Fractals}

In this section we develop a theory of quasisymmetries for finitely ramified fractals with respect to undistorted metrics. We prove that the set of undistorted metrics on a finitely ramified fractal is a quasisymmetry class in Section~\ref{sec:CompatibleMetricProof}.  In Section~\ref{subsec:ProofThm1-2}, we characterize the finitely ramified fractals for which undistorted metrics exist.  Section~\ref{subsec:ExistUndistorted} develops a sufficient condition to establish that a metric on a finitely ramified fractal is undistorted. Lastly, Section~\ref{subsec:PiecewiseCellularAreQuasisymmetries} demonstrates a standard method to create quasisymmetries between finitely ramified fractals equipped with undistorted metrics.

\subsection{Equivalence of undistorted metrics}
\label{sec:CompatibleMetricProof}

We begin with a proof of our main theorem regarding undistorted metrics.

{
\renewcommand{\thetheorem}{\ref{thm:CompatibleMetric}}
\begin{theorem}\textCompaibleMetric
\end{theorem}
\addtocounter{theorem}{-1}
}

If $X$ is a finitely ramified fractal and $(x,y)$ is a pair of distinct points in~$X$, there exists a maximum integer $n$ so that $x$ and $y$ lie an in intersecting pair of $n$-cells $(E_x,E_y)$ with $x \in E_x$ and $y \in E_y$. We refer to such a pair of $n$-cells as a \newword{covering pair} for $(x,y)$.  Note that $E_x$ and $E_y$ need not be distinct. Note also that $E_x$ and $E_y$ are uniquely determined if $x$ and $y$ lie in their interiors, but if $x$ or $y$ is a boundary point of an $n$-cell then the associated covering pair is not necessarily unique.

\begin{lemma}\label{lem:CoveringPairsNewNew}
Let $X$ be a finitely ramified fractal, and let $d$ be an undistorted metric on $X$.  Then there exists a constant $\alpha\geq 1$ such that for every pair $(x,y)$ of distinct points in $X$ and every covering pair $(E_x,E_y)$ for~$(x,y)$, we have
\[
\frac{1}{\alpha} \leq \frac{d(x,y)}{\diam(E_x)} \leq \alpha.
\]
\end{lemma}
\begin{proof}
Suppose $d$ is $(r,R,C,\delta)$-undistorted.  Let $(x,y)$ be a pair of distinct points in~$X$ with covering pair $(E_x,E_y)$ of $n$-cells.  Then $\diam(E_y)\leq C\diam(E_x)$, so
\[
d(x,y) \leq \diam(E_x)+\diam(E_y) \leq (1+C)\diam(E_x).
\]
For the opposite inequality,
let $E_x'$ be a child of $E_x$ that contains~$x$, and let $E_y'$ be a child of $E_y$ that contains~$y$.  Then $E_x'$ and $E_y'$ are disjoint, so
\[
d(x,y) \geq d(E_x',E_y') \geq \delta\, \diam(E_x') \geq \frac{\delta r}{C} \diam(E_x).
\]
Thus $\alpha=\max\bigl(1+C,C/(\delta r)\bigr)$ suffices.
\end{proof}

\begin{proof}[Proof of Theorem~\ref{thm:CompatibleMetric}, Part (1)]
Let $X$ be a finitely ramified fractal and let $d$ and $d'$ be undistorted metrics on $X$. We wish to prove that $d$ and $d'$ are quasi-equivalent.  For any set $S\subseteq X$, let $\diam(S)$ denote its diameter with respect to~$d$, and let $\diam'(S)$ denote its diameter with respect to~$d'$. Taking minimums and maximums as appropriate, we can find constants $0<r<R<1$, $C\geq 1$, and $\delta> 0$ so that both $d$ and $d'$ are $(r,R,C,\delta)$-undistorted.  We can also find a single constant $\alpha\geq 1$ that satisfies Lemma~\ref{lem:CoveringPairsNewNew} for both $d$ and~$d'$.

Let $x,a,b\in X$ be distinct points.  Let $(E_{x,1},E_a)$ be a covering pair for $(x,a)$, let $(E_{x,2},E_b)$ be a covering pair for~$(x,b)$, and let $n_1$ and $n_2$ be the respective levels of these pairs.  Note that $E_{x,1}$ and $E_{x,2}$ intersect since they both contain~$x$.  If $n_1 \geq n_2$, it follows that
\[
\frac{d'(x,a)}{d'(x,b)} \leq \frac{\alpha\, \diam'(E_{x,1})}{\alpha^{-1}\diam'(E_{x,2})} \leq \alpha^2 C R^{n_1-n_2}
\]
and 
\[
\frac{d(x,a)}{d(x,b)} \geq \frac{\alpha^{-1}\diam(E_{x,1})}{\alpha\diam(E_{x,2})} \geq
\frac{r^{n_1-n_2}}{\alpha^2 C}
\]
so
\[
\frac{d'(x,a)}{d'(x,b)} \leq \eta_1\biggl(\frac{d(x,a)}{d(x,b)}\biggr)
\]
where $\eta_1\colon [0,\infty)\to [0,\infty)$ is the homeomorphism 
\[
\eta_1(t)=\alpha^2 C\bigl(\alpha^2C t\bigr)^{\log(R)/\log(r)}.
\]
If instead $n_1 \leq n_2$, then
\[
\frac{d'(x,a)}{d'(x,b)} \leq \frac{\alpha \diam'(E_{x,1})}{\alpha^{-1}\diam'(E_{x,2})} \leq \frac{\alpha^2C}{r^{n_2-n_1}}
\]
and
\[
\frac{d(x,a)}{d(x,b)} \geq \frac{\alpha^{-1}\diam(E_{x,1})}{\alpha\diam(E_{x,2})} \geq
\frac{1}{\alpha^2 CR^{n_2-n_1}}
\]
so
\[
\frac{d'(x,a)}{d'(x,b)} \leq \eta_2\biggl(\frac{d(x,a)}{d(x,b)}\biggr)
\]
where $\eta_2\colon [0,\infty)\to [0,\infty)$ is the homeomorphism 
\[
\eta_2(t)=\alpha^2 C\bigl(\alpha^2C t\bigr)^{\log(r)/\log(R)}.
\]
We conclude that
\[
\frac{d'(x,a)}{d'(x,b)} \leq \eta\biggl(\frac{d(x,a)}{d(x,b)}\biggr)
\]
for all triples of distinct points $a,b,x\in X$, where $\eta\colon [0,\infty)\to [0,\infty)$ is the homeomorphism $\eta(t)=\max(\eta_1(t),\eta_2(t))$.
\end{proof}

For the proof of part (2) of Theorem~\ref{thm:CompatibleMetric}, we need a couple of lemmas that will also be useful later.  The first is a subtly powerful characterization of undistorted metrics.  Because the cell separation condition requires comparing the diameter of an $n$-cell to the diameter of every other $n$-cell, the definition of an undistorted metric is inherently global in nature. On the other hand, the characterization below involves only local comparisons, comparing the diameter of a cell $E \in \E_n$ to the diameters of cells contained $n-1$ cells intersecting the parent of $E$.

\begin{lemma}\label{lem:EquivalentExponentialDecay}Let $X$ be a finitely ramified fractal.  A continuous metric $d$ on $X$ is undistorted if and only if it satisfies the following conditions:
\begin{enumerate}
    \item There exists a constant $\lambda \geq 1$ so that
    \[
    \frac{\diam(E_1)}{\diam(E_2)} \leq \lambda
    \]
    for any two $n$-cells $E_1$ and $E_2$ that intersect.\smallskip
    \item There exists a constant $\mu > 0$ so that
    \[    \diam(E') \geq \mu \diam(E)
    \]
    for any $n$-cell $E$ and any $(n+1)$-cell $E'$ contained in~$E$.\smallskip
    \item There exist constants $k\in\mathbb{N}$ and\/ $0<\nu<1$ so that
    \[
    \diam(E') \leq \nu \diam(E)
    \]
    for any $n$-cell $E$ and any $(n+k)$-cell $E'$ contained in~$E$.\smallskip
    \item There exists a constant $\delta>0$ so that
\[
d(E_1,E_2)\geq \delta \diam(E_1)
\]
for every pair of disjoint $n$-cells $E_1,E_2$ whose parents intersect.
\end{enumerate}
\end{lemma}
\begin{proof}
We will prove that conditions (1), (2), and (3) are together equivalent to the exponential decay condition, while condition (4) is equivalent to the cell separation condition.

If $d$ satisfies the exponential decay condition with constants $(r,R,C)$, then we can choose a $k\in \mathbb{N}$ so that $CR^k<1$, in which case $d$ satisfies conditions (1), (2), and (3) with $\lambda = C$, $\mu=r/C$, and $\nu=CR^k$.

Conversely, suppose $d$ satisfies (1), (2), and (3) with constants $(\lambda,\mu,k,\nu)$, and let $E$ be an $m$-cell and $E'$ an $n$-cell that intersect, where $m\leq n$. For the lower bound, let $E_0'$ be an $m$-cell containing $E'$.  Then
\[
\frac{\diam(E')}{\diam(E)} \geq \frac{\mu^{n-m}\diam(E_0')}{\diam(E)} \geq \frac{\mu^{n-m}}{\lambda}.
\]
For the upper bound, let $q$ be an integer so that $qk\leq n-m < (q+1)k$, and let $E_1'$ be an $(m+qk)$-cell so that $E'\subseteq E_1'\subseteq E_0'$.  Then
\[
\frac{\diam(E')}{\diam(E)} \leq \frac{\diam(E_1')}{\diam(E)} \leq \frac{\nu^q \diam(E_0')}{\diam(E)} \leq \nu^q \lambda
\leq \frac{\lambda}{\nu}\bigl(\nu^{1/k}\bigr)^{n-m}.
\]
Thus $d$ satisfies the inequality for the exponential decay condition with constants $r=\mu$, $R=\nu^{1/k}$, and $C=\max(\lambda,\lambda/\nu)=\lambda/\nu$.  Note that $R<1$ since $\nu < 1$, and it follows automatically that $r\leq R$.

Finally, any metric that satisfies the cell separation condition for some constant $\delta>0$ clearly satisfies condition~(4).  For the converse, suppose $d$ satisfies condition~(4) for some constant $\delta>0$, and let $E_1$ and $E_2$ be any two disjoint $n$-cells. Then there exists a maximum integer $m$ so that $E_1\subseteq E_1'$ and $E_2\subseteq E_2'$ for some $m$-cells $E_1'$ and $E_2'$ that intersect.  Let $E_1''$ and $E_2''$ be $(m+1)$-cells so that $E_1\subseteq E_1''\subseteq E_1'$ and $E_2\subseteq E_2''\subseteq E_2'$.  Then $E_1''$ and $E_2''$ are disjoint but their parents intersect, so
\[
d(E_1,E_2) \geq d(E_1'',E_2'') \geq \delta \diam(E_1'') \geq \delta \diam(E_1).\qedhere
\]
\end{proof}

The following basic lemma about quasisymmetries will also be useful later.
  
\begin{lemma}\label{lem:DiameterRatioGeneral}
Let $X$ and $Y$ be metric spaces, and let $h \colon X \to Y$ be an $\eta$\mbox{-}quasi\-symmetry for some homeomorphism $\eta\colon[0,\infty)\to[0,\infty)$. If $S, T \subseteq X$ and $S \cap T \ne \emptyset$, then
\[
\frac{\diam(h(S))}{\diam(h(T))} \leq 2\,\eta\biggl(2\frac{\diam(S)}{\diam(T)}\biggr).
\]
\end{lemma}
\begin{proof}
Fix a point $p \in S\cap T$.  Then there must exist a point $t \in T$ such that $d(p,t)\geq \diam(T)/2$. For any point $s \in S$ it follows that
\[
\frac{d(p,s)}{d(p,t)} \leq \frac{\diam(S)}{\diam(T)/2} = 2\frac{\diam(S)}{\diam(T)}
\]
so
\begin{multline*}
d(h(p),h(s)) = \frac{d(h(p),h(s))}{d(h(p),h(t))}\,d(h(p),h(t)) \\
\leq \eta\biggl(\frac{d(p,s)}{d(p,t)}\biggr)\diam(h(T)) \leq \eta\biggl(2\frac{\diam(S)}{\diam(T)}\biggr) \diam(h(T)).
\end{multline*}
Since this holds for all $s\in S$, it follows that
\[
\diam(h(S)) \leq 2\,\eta\biggl(2\frac{\diam(S)}{\diam(T)}\biggr) \diam(h(T)).\qedhere
\]
\end{proof}

\begin{proof}[Proof of Theorem~\ref{thm:CompatibleMetric}, Part (2)]
Let $X$ be a finitely ramified fractal, let $d$ and $d'$ be metrics on~$X$, where $d$ is $(r,R,C,\delta)$-undistorted, and suppose that the identity map $(X,d)\to (X,d')$ is $\eta$-quasisymmetric for some homeomorphism $\eta\colon [0,\infty)\to[0,\infty)$. We wish to prove that $d'$ is undistorted.  For any set $S\subseteq X$, let $\diam(S)$ and $\diam'(S)$ denote its diameters with respect to $d$ and~$d'$. It suffices to verify that $d'$ satisfies the conditions in Lemma~\ref{lem:EquivalentExponentialDecay}.

For condition~(1) in Lemma~\ref{lem:EquivalentExponentialDecay}, let $E_1$ and $E_2$ be $n$-cells in $X$ that intersect.  Then by Lemma~\ref{lem:DiameterRatioGeneral} we have
\[
\frac{\diam'(E_1)}{\diam'(E_2)} \leq 2\,\eta\biggl(2\frac{\diam(E_1)}{\diam(E_2)}\biggr) \leq 2\,\eta(2C)
\]
so $\lambda=2\,\eta(2C) \geq 1$ suffices.  

For condition~(2) in Lemma~\ref{lem:EquivalentExponentialDecay}, let $E$ be an $n$-cell and $E'$ an $(n+1)$-cell contained in~$E$.  By Lemma~\ref{lem:DiameterRatioGeneral},
\[
\frac{\diam'(E)}{\diam'(E')} \leq 2\,\eta\biggl(2\frac{\diam(E)}{\diam(E')}\biggr) \leq 2\,\eta\biggl(2\frac{C}{r}\biggr)
\]
so $\mu=\bigl(2\,\eta(2C/r)\bigr)^{-1}$ suffices.

For condition~(3) in Lemma~\ref{lem:EquivalentExponentialDecay}, observe that since $\eta\colon [0,\infty)\to[0,\infty)$ is a homeomorphism and $0<R<1$ there exists a $k\in\mathbb{N}$ so that $\eta(2CR^k)<1/2$.  If $E$ is an $n$-cell and $E'$ is an $n+k$ cell that is contained in~$E$, then by Lemma~\ref{lem:DiameterRatioGeneral},
\[
\frac{\diam'(E')}{\diam'(E)} \leq 2\,\eta\biggl(2\frac{\diam(E')}{\diam(E)}\biggr) \leq 2\,\eta(2CR^k) < 1
\]
so $\nu=2\,\eta(2CR^k)$ suffices.

For condition~(4) in Lemma~\ref{lem:EquivalentExponentialDecay}, let $E_1$ and $E_2$ be disjoint $n$-cells whose parents $\hat{E}_1,\hat{E}_2$ intersect.  Choose points $p_1\in E_1$ and $p_2\in E_2$ so that $d'(p_1,p_2)=d'(E_1,E_2)$.  Since 
\[
\frac{\diam(E_2)}{\diam(E_1)} \leq \frac{CR\diam(\hat{E}_2)}{(r/C)\diam(\hat{E}_1)} \leq \frac{C^3R}{r}
\]
it follows from Lemma~\ref{lem:DiameterRatioGeneral} that
\begin{multline*}
\frac{d'(E_1,E_2)}{\diam'(E_1)} = \frac{\diam'(\{p_1,p_2\})}{\diam'(E_1)} \leq 2\,\eta\biggl(2\frac{\diam(\{p_1,p_2\})}{\diam(E_1)}\biggr) \\[6pt]
\leq 2\,\eta\biggl(2\frac{d(E_1,E_2)+\diam(E_1)+\diam(E_2)}{\diam(E_1)}\biggr) \leq 2\,\eta\biggl(2\biggl(\delta+1+\frac{C^3R}{r}\biggr)\biggr).\quad\qedhere
\end{multline*}
\end{proof}

\subsection{Existence of undistorted metrics}
\label{subsec:ProofThm1-2}

In this section we prove Theorem~\ref{thm:AdmitsUndistortedMetric} from the introduction, which characterizes the finitely ramified fractals that admit undistorted metrics.  Though we will not actually use this theorem later on, the fact that finitely ramified fractals commonly admit such a metric is important motivation for all of our work.

{
\renewcommand{\thetheorem}{\ref{thm:AdmitsUndistortedMetric}}
\begin{theorem}\textAdmitsUndistortedMetric
\end{theorem}
\addtocounter{theorem}{-1}
}

Both of the conditions in this theorem are necessary.  It is easy to construct finitely ramified cell structures on the interval $[0,1]$, for example, that do not satisfy condition (1), and therefore do not admit an undistorted metric.  There are also finitely ramified fractals that satisfy condition~(1) but not condition~(2), as the following example shows.

\begin{example}\label{ex:NoUndistorted}
Let $X$ be the union of the line segments $[-1,1]\times \{0\}$ and $  \{0\}\times [0,1]$ in~$\R^2$.  For each $n\geq 1$, let $\delta_n$ denote the $n$th term in the sequence
\[
1,\;\;1/2,\;\;1/2,\;\;1/4,\;\;1/4,\;\;1/4,\;\;1/4,\;\;1/8,\;\;\ldots
\]
i.e.\ $\delta_n=2^{-\lfloor\log_2(n)\rfloor}$. Let 
\[
V_n=\bigl\{(x,y)\in X \;\bigr|\; x,y\in 2^{-n}\Z\text{ and  either }|x|\geq \delta_n\text{ or }y>0\bigr\},
\]
and let $\mathcal{E}_n$ be the set of closures of the connected components of $X\setminus V_n$.  Then $\{\mathcal{E}_n\}$ defines a finitely ramified cell structure on $X$ (with $\mathcal{E}_0=\{X\}$). Note that every $n$-cell in $X$ is the union of at least two $(n+1)$-cells, and contains a pair of disjoint $(n+2)$-cells.  However, for $j\geq 0$ the $(2^{j+1}-1)$-cell containing the origin intersects both of the $2^j$-cells
\[
\bigl[-2^{-j}-2^{-2^j},-2^{-j}\bigr]\times\{0\}
\qquad\text{and}\qquad
\bigl[2^{-j},2^{-j}+2^{-2^j}\bigr]\times\{0\}.
\]
These $2^j$-cells are disjoint, so by Theorem~\ref{thm:AdmitsUndistortedMetric} this finitely ramified fractal does not admit an  undistorted metric.
\end{example}

We now turn to the proof of Theorem~\ref{thm:AdmitsUndistortedMetric}, which occupies the remainder of this section.  We begin with the forward direction.

\begin{lemma}Let $X$ be a finitely ramified fractal.  If $X$ admits an undistorted metric, then $X$ satisfies conditions (1) and (2) in Theorem~\ref{thm:AdmitsUndistortedMetric}.
\end{lemma}
\begin{proof}
Let $d$ be an $(r,R,C,\delta)$-undistorted metric on~$X$, and let $k\in\N$ so that $CR^k < \min(\delta,1/4)$.

For condition (1), let $E$ be an $n$-cell in $X$, and let $p,q\in E$ so that $d(p,q) \geq \diam(E)/2$.  Let $E_1$ and $E_2$ be level $(n+k)$ descendants of $E$ that contain $p$ and $q$ respectively.  Then
\[
\diam(E_i)\leq CR^k\diam(E) < \diam(E)/4
\]
for each~$i$, so $E_1$ and $E_2$ must be disjoint.

For condition (2), let $E_1$ and $E_2$ be disjoint $n$-cells in $X$, and observe that $d(E_1,E_2) \geq \delta\,\diam(E_1)$ by the cell separation condition.  If $E$ is any $(n+k)$-cell that intersects~$E_1$, then
\[
\diam(E) \leq CR^k\diam(E_1) < \delta\diam(E_1)
\]
and therefore $E$ cannot intersect~$E_2$.
\end{proof}

For the converse, we must begin with a finitely ramified fractal that satisfies conditions (1) and (2) in the statement of Theorem~\ref{thm:AdmitsUndistortedMetric} and construct an undistorted metric on $X$. If $p$ and $q$ are points in $X$, define a \newword{cell chain} connecting $p$ and $q$ to be a sequence $E_1,\ldots,E_m$ of cells such that $p\in E_1$, $q\in E_m$, and each $E_i\cap E_{i+1}\ne \emptyset$.  Given any $\alpha>1$, let
\[
d_\alpha(p,q) = \inf_{E_1,\ldots,E_l} \sum_{i=1}^l \alpha^{-|E_i|}
\]
where $|E_i|$ denotes the level of $E_i$, and the infimum is taken over all cell chains connecting $p$ and~$q$.

For the following lemma, observe that if $p$ and $q$ are distinct points in $X$, then there exists an $n\in\N$ so that $p$ and $q$ are not contained in any intersecting pair of $n$-cells (see~\cite[Proposition~2.9]{Tep}).  We will denote the smallest such $n$ by $P(p,q)$.

\begin{lemma}\label{lem:ItIsAMetric}
Let $X$ be a finitely ramified fractal that satisfies condition~(2) of Theorem~\ref{thm:AdmitsUndistortedMetric} with respect to some $k\in\N$, and let $1< \alpha \leq (3/2)^{1/k}$.  Then $d_\alpha$ is a metric on $X$, and 
\[
d(p,q) \geq \alpha^{1-P(p,q)-k}
\]
for any distinct points $p,q\in X$.
\end{lemma}
\begin{proof}
It is immediate by the definition that $d_\alpha$ is symmetric, satisfies the triangle inequality, and that $d_\alpha(p,p)=0$ for all $p\in X$.  Thus, to prove that $d_\alpha$ is a metric we need only show that $d_\alpha(p,q) > 0$ for all distinct $p$ and $q$, which will follow from the stated inequality.

To prove the inequality, let $p$ and $q$ be distinct points in $X$, and let $n=P(p,q)$.  We must show that $\sum_{i=1}^l \alpha^{-|E_i|} \geq \alpha^{1-n-k}$ for any cell chain $E_1,\ldots,E_l$ connecting $p$ to $q$.  This holds automatically if $|E_1|< n$ or $|E_l|<n$, so we will assume throughout that $|E_1|\geq n$ and $|E_l|\geq n$. For such a chain, we will prove by induction that
\[
\sum_{1<i<l} \alpha^{-|E_i|} \geq \alpha^{1-n-k}.
\]
Note that this sum excludes $i=1$ and $i=l$.

The base case is $l \leq 3$.  For this case, since $|E_1|,|E_3| \geq n$, there exists $n$-cells $E_1'\supseteq E_1$ and $E_3'\supseteq E_3$.  Since $p\in E_1'$ and $q\in E_3'$, these $n$-cells are disjoint, so $l$ must be~$3$.  If $|E_2|\geq n+k$, then $E_2$ is contained in some $(n+k)$-cell that intersects both $E_1'$ and $E_3'$, contradicting condition (2) of Theorem~\ref{thm:AdmitsUndistortedMetric}.  We conclude that $|E_2|< n+k$, and hence $\sum_{1<i<l}\alpha^{-|E_i|} = \alpha^{-|E_2|} \geq \alpha^{1-n-k}$.

For the induction step, suppose $E_1,\ldots, E_l$ is a cell chain of length $l\geq 4$ connecting $p$ and $q$ for which $|E_1|,|E_l|\geq n$. Replacing $E_1$ and $E_l$ with $n$-cells that contain them, we may assume that $|E_1|=|E_l|=n$.  Choose a cell $E_j$ from $E_2,\ldots,E_{l-1}$ whose  level $M=|E_j|$ is a large as possible.  If $M<n+k$ we are done, so suppose $M\geq n+k$.

Consider first the case where $|E_{j-1}|,|E_{j+1}|\leq M-k$.  Since $E_{j-1}$ and $E_{j+1}$ intersect $E_j$, we can find $(M-k)$-cells $E_{j-1}'\subseteq E_{j-1}$ and $E_{j+1}'\subseteq E_{j+1}$ that intersect~$E_j$.  By condition (2) of Theorem~\ref{thm:AdmitsUndistortedMetric}, no two disjoint $(M-k)$-cells can intersect a common $M$-cell, so $E_{j-1}'$ and $E_{j+1}'$ cannot be disjoint.  We conclude that $E_{j-1}$ and $E_{j+1}$ intersect, so $E_1,\ldots,E_{j-1},E_{j+1},\ldots E_l$ is a cell chain connecting $p$ and $q$.  By our induction hypothesis, it follows that
\[
\sum_{1<i<l} \alpha^{-|E_i|} \geq \sum_{\substack{1 < i < l \\ i\ne j}} \alpha^{-|E_i|} \geq \alpha^{1-n-k}.
\]

All that remains is the case where  $|E_{j-1}|> M-k$ or $|E_{j+1}|> M-k$.  Let $E_{j-1}'$ be a cell containing $E_{j-1}$ at level $|E_{j-1}'|=\min(|E_{j-1}|,M-k)$, let $E_{j+1}'$ be a cell containing $E_{j+1}$ at level $|E_{j+1}'| = \min(|E_{j+1}|,M-k)$, and let $E_i'=E_i$ for $i\ne j\pm 1$.  Then $E_1',\ldots, E_l'$ is again a cell chain connecting $p$ and $q$.  Note that $\min(n,M-k)=n$, so $E_1'$ and $E_l'$ are both $n$-cells, even in the case where $j=2$ or~$j=l-1$.  Since $|E_j'| = M$ and $|E_{j\pm 1}'|\leq M-k$, it follows from the argument in the previous paragraph that $E_{j-1}'$ and $E_{j+1}'$ intersect, so $E_1',\ldots,E_{j-1}',E_{j+1}',\ldots,E_{l}'$ is a cell chain.  By our induction hypothesis, we conclude that
\[
\sum_{\substack{1<i<l \\ i\ne j}} \alpha^{-|E_i'|}\geq \alpha^{1-n-k}.
\]
But
\[
\sum_{1<i<l} \alpha^{-|E_i|} - \sum_{\substack{1<i<l \\ i\ne j}} \alpha^{-|E_i'|} = \Bigl(\alpha^{-|E_{j-1}|} - \alpha^{-|E_{j-1}'|}\Bigr) + \alpha^{-|E_j|} + \Bigl(\alpha^{-|E_{j+1}|} - \alpha^{-|E_{j+1}'|}\Bigr)
\]
so it suffices to prove that the expression on the right is non-negative.

Now, if $|E_{j-1}| > M-k$, then $|E_{j-1}'|=M-k$.  Since $|E_{j-1}|\leq M$, it follows that
\[
\alpha^{-|E_{j-1}|} - \alpha^{-|E_{j-1}'|} \geq \alpha^{-M} - \alpha^{k-M}.
\]
Since $\alpha>1$, the expression on the right is negative, so this inequality also holds when $|E_{j-1}|\leq M-k$, since $E_{j-1}'=E_{j-1}$ in that case.  Of course, a similar inequality also holds for $j+1$ instead of $j-1$.  Since $|E_j|=M$, we conclude that
\begin{align*}
\sum_{1<i<l} \alpha^{-|E_i|} - \sum_{\substack{1<i<l \\ i\ne j}} \alpha^{-|E_i'|}
&\geq
\Bigl(\alpha^{-M}-\alpha^{k-M}\Bigr) + \alpha^{-M} + \Bigl(\alpha^{-M}-\alpha^{k-M}\Bigr) \\
&= \Bigl(3 - 2\alpha^k\Bigr) \alpha^{-M} \geq 0
\end{align*}
where the last step uses the hypothesis that $1<\alpha\leq (3/2)^{1/k}$.  By induction, we conclude that $\sum_{1<i<l} \alpha^{-|E_i|}\geq \alpha^{1-n-k}$ for any cell chain connecting $p$ and $q$ as long as $|E_1|,|E_l|\geq n$, and the desired inequality follows.
\end{proof}

The following lemma completes the proof of Theorem~\ref{thm:AdmitsUndistortedMetric}.

\begin{lemma}\label{lem:ConstructedMetricIsUndistorted}
Let $X$ be a finitely ramified fractal that satisfies conditions (1) and~(2) of Theorem~\ref{thm:AdmitsUndistortedMetric} with respect to some $k\in\N$, and let $1< \alpha \leq (3/2)^{1/k}$.  Then $d_\alpha$ is an undistorted metric on~$X$.
\end{lemma}
\begin{proof}
By Lemma~\ref{lem:ItIsAMetric}, the function $d_\alpha$ is a metric on~$X$.  For the exponential decay condition, let $E$ be any $n$-cell in~$X$. Then $E$ is itself a cell chain connecting any two points in~$E$, so $\diam(E)\leq \alpha^{-n}$.  For a lower bound on the diameter of~$E$, observe that by condition~(1) of Theorem~\ref{thm:AdmitsUndistortedMetric} we can find disjoint $(n+k)$-cells $E_1$ and $E_2$ that are contained in $E$.  If $p$ and $q$ are interior points of $E_1$ and $E_2$, respectively, then $E_1$ and $E_2$ are the only $(n+k)$-cells that contain $p$ and $q$, so $P(p,q)\leq n+k$.  By Lemma~\ref{lem:ItIsAMetric}, it follows that $d(p,q)\geq \alpha^{1-n-2k}$.  Thus
\[
\alpha^{1-n-2k} \leq \diam(E) \leq \alpha^{-n}.
\]
Since $k$ is a constant, the exponential decay condition follows easily.

For the cell separation condition, observe that each cell has at most finitely many boundary points, namely its intersection points with other cells at the same level, and the remaining points are interior points.  Since each cell is compact, connected, and has nonempty interior, the interior points of any cell are dense in the cell. Let $E_1$ and $E_2$ be disjoint $n$-cells, and let $p_1$ and $p_2$ be interior points of these cells.  Then $P(p_1,p_2)\leq n$, so $d(p_1,p_2)\geq \alpha^{1-n-k}$ by Lemma~\ref{lem:ItIsAMetric}.  Since interior points are dense, we conclude that $d(E_1,E_2)\geq \alpha^{1-n-k}$. But $\diam(E_1)\leq \alpha^{-n}$ by the argument in the previous paragraph, so $d(E_1,E_2)\geq \alpha^{1-k}\diam(E_1)$, which verifies the cell separation condition.
\end{proof}

\begin{remark}\label{rem:StrongExponentialDecay}
In Lemma~\ref{lem:ConstructedMetricIsUndistorted}, we actually proved a stronger version of the exponential decay condition, namely:
\begin{description}
    \item[Strong Exponential Decay] There exist constants $0<r<1$ and $C\geq 1$ so that
\[
\frac{r^n}{C} \leq \diam(E) \leq Cr^n
\]
for every $n$-cell~$E$.
\end{description}
Not every undistorted metric satisfies this condition (see Example~\ref{ex:WeirdIntervalCells}), and indeed such metrics seem to be rare in the context of Julia sets.  However, it follows from  Lemma~\ref{lem:ConstructedMetricIsUndistorted} and Theorem \ref{thm:CompatibleMetric} that every undistorted metric on a finitely ramified fractal is quasisymmetrically equivalent to an undistorted metric that satisfies the strong exponential decay condition.  \end{remark}

\begin{figure}
    \centering
    \includegraphics{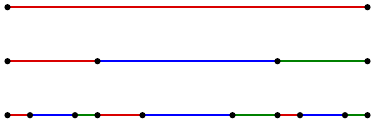}
    \caption{A finitely ramified cell structure on $[0,1]$ with one \mbox{$0$-cell}, three $1$-cells, nine $2$-cells, and so forth.}
    \label{fig:WeirdIntervalCells}
\end{figure}

\begin{example}\label{ex:WeirdIntervalCells}
 Consider the cell structure on the interval $[0,1]$ shown in Figure~\ref{fig:WeirdIntervalCells}, where each $n$-cell $[a,b]$ is subdivided into three $(n+1)$-cells of lengths $(b-a)/4$, $(b-a)/2$, and $(b-a)/4$, respectively.  It is not hard to check that these cells satisfy the hypotheses of
Lemma~\ref{lem:EquivalentExponentialDecay} with $(\lambda,\mu,k,\nu,\delta)=(2,1/4,1,1/2,1/2)$, so the Euclidean metric is undistorted with respect to this cell structure.  However, the strong exponential decay condition is not satisfied, since for each $n\geq 1$ there are $n$-cells of diameters $2^{-n}$ and $4^{-n}$.
\end{example}

\subsection{Sufficient condition for undistorted metrics}\label{subsec:ExistUndistorted}
Though a finitely ramified fractal must be fragmented at progressively smaller scales, the definition does not include any requirement of self-similarity. In this section, we develop a notion of quasi-self-similarity for a metric on a finitely ramified fractal and show that any such metric is undistorted.  

One basic way to enforce self-similarity for a finitely ramified fractal is to require that there be finitely many different types of cells, where two cells have the same type if there exists a cellular homeomorphism between them.  Any such fractal arises as the limit space of a colored hyperedge\footnote{A hyperedge replacement system is exactly the analog of an edge replacement system as described in \cite{BeFo2}, except that we use hypergraphs instead of graphs, with an ordering on the vertices of each hyperedge, and all hyperedges of the same color having the same number of vertices.} replacement system, as described in~\cite{BeFo2}. Such a system has one color for each type of cell, with expansion rules describing the subdivision of a cell into its children.

The following proposition considers a fractal with finitely many types of cells, with the added condition that the cellular homeomorphisms are uniform quasisymmetries between cells and between \newword{neighboring pairs} of cells, i.e.\ pairs of distinct cells at the same level that intersect. 
 We will use this proposition in Section~\ref{subsec:PiecewiseCanonical} to prove that the Euclidean metric on certain Julia sets is undistorted.

\begin{proposition} \label{prop:undistorted}
Let $X$ be a finitely ramified fractal, and let $d$ be a continuous metric on $X$.  Suppose that there exists a finite collection $\M$ of cells in $X$ and a homeomorphism $\eta\colon[0,\infty)\to[0,\infty)$ with the following properties:
\begin{enumerate}
    \item For every cell $E$ in $X$, there exists an $M\in \M$ and a cellular $\eta$-quasisymmetry $E\to M$.\smallskip
    \item For every pair $E_1,E_2$ of neighboring cells in $X$ there exists a pair $M_1,M_2$ of neighboring cells in $\M$ and a cellular $\eta$-quasisymmetry $E_1\cup E_2\to M_1\cup M_2$.
\end{enumerate}
Then the metric $d$ is undistorted.
\end{proposition}
\begin{proof}Let $\mathcal{M}$ be a finite collection of cells and $\eta\colon [0,\infty)\to [0,\infty)$ a homeomorphism satisfying conditions (1) and (2) above. It suffices to verify that $d$ satisfies the conditions in Lemma~\ref{lem:EquivalentExponentialDecay}.

Note that the inverse of an $\eta$-quasisymmetry is itself an $\eta'$-quasisymmetry with $\eta'(t)=1/\eta^{-1}(1/t)$ for $t > 0$. By replacing $\eta$ with $\max(\eta,\eta')$, we can assume that the inverses of the $\eta$-quasisymmetries in conditions (1) and (2) above are also $\eta$-quasisymmetries.

For condition~(1) in Lemma~\ref{lem:EquivalentExponentialDecay}, let $\lambda_0$ be the maximum value of the ratio $\diam(M_1)/\diam(M_2)$ for every intersecting pair of cells $M_1,M_2\in\mathcal{M}$. Let $E_1$ and $E_2$ be neighboring cells, and let $M_1$ and $M_2$ be neighboring cells in $\M$ that map to $E_1$ and $E_2$ by a cellular $\eta$-quasisymmetry.  By Lemma~\ref{lem:DiameterRatioGeneral}, we have
\[
\frac{\diam(E_1)}{\diam(E_2)} \leq 2\,\eta\biggl(2\frac{\diam(M_1)}{\diam(M_2)}\biggr) \leq 2\,\eta(2\lambda_0)
\]
so $\lambda= 2\,\eta(2\lambda_0) \geq 1$ suffices.

For condition~(2) in Lemma~\ref{lem:EquivalentExponentialDecay}, let $\mu_0$ be the minimum value of the ratio $\diam(M')/\diam(M)$ as $M$ ranges over all cells in $\M$ and $M'$ ranges over all children of $M$. Let $E$ be a cell and $E'$ a child of~$E$, and let $M$ and $M'$ be cells in $\M$ that map to $E$ and $E'$ by a cellular $\eta$-quasisymmetry. By Lemma~\ref{lem:DiameterRatioGeneral},
\[
\frac{\diam(E)}{\diam(E')} \leq 2\,\eta\biggl(2\frac{\diam(M)}{\diam(M')}\biggr) \leq 2\,\eta(2\mu_0^{-1})
\]
so $\mu= \bigl(2\,\eta(2\mu_0^{-1})\bigr)^{-1}$ suffices.

For condition~(3) in Lemma~\ref{lem:EquivalentExponentialDecay}, note that by condition (4) of Definition \ref{def:finitelyramifiedcellstructure} there exists a $k\in\N$ so that
\[
\frac{\diam(M')}{\diam(M)}\leq \frac{1}{2}
\]
for each $n$-cell $M\in\mathcal{M}$ and each $(n+k)$-cell $M'$ contained in~$M$.  Now let $E$ be an arbitrary $n$-cell and $E'$ an $(n+k)$-cell contained in~$E$, and let $M$ and $M'$ be cells in $\M$ that map to $E$ and $E'$ by a cellular $\eta$-quasisymmetry. By Lemma~\ref{lem:DiameterRatioGeneral},
\[
\frac{\diam(E')}{\diam(E)} \leq 2\,\eta\biggl(2\frac{\diam(M')}{\diam(M)}\biggr) \leq 2\,\eta(1)
\]
so $\nu= 2\,\eta(1)$ together with $k$ suffice.

For condition~(4) in Lemma~\ref{lem:EquivalentExponentialDecay}, let $\delta_0$ be the maximum of
\[
\frac{d(M_1,M_2) + \diam(M_2)}{\diam(M_1)}
\]
as $M_1$ and $M_2$ range over all disjoint cells whose parents are equal or neighboring cells in~$\M$.  Now let $E_1$ and $E_2$ be disjoint cells whose parents are equal or neighboring, and choose points $p_1\in E_1$ and $p_2\in E_2$ so that $d(p_1,p_2)=d(E_1,E_2)$.  Let $M_1$ and $M_2$ be cells whose parents are equal or neighboring cells in $\M$ so that $M_1$ and $M_2$  map to $E_1$ and $E_2$ by a cellular $\eta$-quasisymmetry, and let $q_1$ and $q_2$ be the points that map to $p_1$ and $p_2$. It follows from Lemma~\ref{lem:DiameterRatioGeneral} that
\begin{multline*}
\frac{d(E_1,E_2)}{\diam(E_1)} = \frac{\diam(\{p_1,p_2\})}{\diam(E_1)} \leq 2\,\eta\biggl(2\frac{\diam(\{q_1,q_2\})}{\diam(M_1)}\biggr) \\[6pt]
\leq 2\,\eta\biggl(2\frac{d(M_1,M_2)+\diam(M_1)+\diam(M_2)}{\diam(M_1)}\biggr)
\leq 2\,\eta(2(\delta_0+1))\quad
\end{multline*}
so $\delta= 2\,\eta(2(\delta_0+1))$ suffices.  
%Since
%\[
%\frac{\diam(E_2)}{\diam(E_1)} \leq \frac{\diam(E_2')}{(1/\mu)\diam(E_1')} \leq \mu\lambda
%\]
\end{proof}

\subsection{Quasisymmetries between finitely ramified fractals}\label{subsec:PiecewiseCellularAreQuasisymmetries}
In this section, we apply Corollary \ref{cor:PullbackTest} to piecewise cellular homeomorphisms between finitely ramified fractals.  In particular, we show that any piecewise cellular homeomorphism $h\colon X\to Y$ between finitely ramified fractals with undistorted metrics is a quasisymmetry. In the case where $X=Y$ this often produces a large class of quasisymmetries, and we exploit this idea in Theorems \ref{thm:BubbleBathModularGroup}, \ref{thm:UnicriticalTheorem}, and \ref{thm:ContainsF} to build many quasisymmetries of Julia sets.

{
\renewcommand{\thetheorem}{\ref{thm:PiecewiseCellular}}
\begin{theorem}\textPiecewiseCellular
\end{theorem}
\addtocounter{theorem}{-1}
}
\begin{proof}
Suppose the metric $d$ on $Y$ is $(r,R,C,\delta)$-undistorted, and let $h\colon X\to Y$ be a piecewise cellular homeomorphism. Then there exists an $M\in\N$ so that $h$ is cellular on all $M$-cells.  Moreover, since there are finitely many such cells, there exists a $j\in \N$ so that each $M$-cell maps to a cell of level between $M-j$ and $M+j$ under $h$.

By Corollary~\ref{cor:PullbackTest}, it suffices to prove that the pullback metric $h^*(d)$ on $X$ is undistorted.  We will show that this metric satisfies the four conditions of Lemma~\ref{lem:EquivalentExponentialDecay}. Note that these conditions are automatically satisfied by any finite collection of cells, so it suffices to check these conditions for $m$-cells in $X$ satisfying $m>M$.  Note that the diameter of any cell $E$ in $X$ with respect to $h^*(d)$ is simply the diameter of $h(E)$ in~$Y$, and similarly the distance between any two cells $E_1,E_2$ in $X$ with respect to $h^*(d)$ is simply the distance between $h(E_1)$ and $h(E_2)$ in~$Y$.

For condition (1) in Lemma~\ref{lem:EquivalentExponentialDecay}, suppose $E_1$ and $E_2$ are intersecting $m$-cells in~$X$ for some $m> M$. Then $h$ maps $E_1$ and $E_2$ to cells $E_1'$ and $E_2'$, respectively, in~$Y$.  Note that $E_1'$ and $E_2'$ intersect and both have level between $m-j$ and $m+j$.  Then
\[
\frac{\diam h(E_1)}{\diam h(E_2)} = \frac{\diam(E_1')}{\diam(E_2')} \leq \max\biggl(CR^{2j},\frac{C}{r^{2j}}\biggr) = Cr^{-2j}
\]
so $\lambda=Cr^{-2j}$ suffices for all such pairs of cells.

For condition~(2) in Lemma~\ref{lem:EquivalentExponentialDecay}, let $E$ in $X$ be an $m$-cell for some $m>M$ and let $E'$ be an $(m+1)$-cell contained in~$E$.  Then there exists an $n\in\N$ so that $h(E)$ is an $n$-cell and $h(E')$ is an $(n+1)$-cell contained in $h(E)$, so
\[
\diam h(E') \geq \frac{r}{C} \diam h(E).
\]
Thus $\mu=r/C$ suffices for all such pairs of cells.

For condition (3) in Lemma~\ref{lem:EquivalentExponentialDecay}, choose an $k\in\N$ so that $CR^k<1$.  Let $E$ be an $m$-cell in $X$ with $m>M$ and let $E'$ be an $(m+k)$-cell contained in $E$. Then there exists an $n\in\N$ so that $h(E)$ is an $n$-cell and $h(E')$ is an $(n+k)$-cell contained in~$h(E)$, so
\[
\frac{\diam h(E')}{\diam h(E)} \leq CR^k < 1.
\]
Thus $\nu=CR^k$ suffices for all such pairs of cells.

Finally, for condition (4) in Lemma~\ref{lem:EquivalentExponentialDecay}, let $E_1$ and $E_2$ be disjoint $m$-cells in $X$, for some $m>M$, whose parents intersect.  Then $h(E_1)$ is an $i_1$-cell and $h(E_2)$ is an $i_2$-cell for some $i_1$ and $i_2$ between $m-j$ and $m+j$.  If $i_1\geq i_2$, let $E_2'$ be an $i_1$-cell that is contained in~$h(E_2)$ so that $d\bigl(h(E_1),E_2'\bigr)=d\bigl(h(E_1),h(E_2)\bigr)$. Then $h(E_1)$ and $E_2'$ are disjoint $i_1$-cells, so
\[
d\bigl(h(E_1),h(E_2)\bigr) = d\bigl(h(E_1),E_2'\bigr) \geq \delta \diam h(E_1).
\]
If instead $i_1\leq i_2$, let $E_1'$ be an $i_2$-cell contained in $h(E_1)$ so that $d\bigl(E_1',h(E_2)\bigr)=d\bigl(h(E_1),h(E_2)\bigr)$. Then $E_1'$ and $h(E_2)$ are disjoint $i_2$-cells, and since $i_2-i_1\leq 2j$ we have
\begin{multline*}
d\bigl(h(E_1),h(E_2)\bigr) = d\bigl(E_1',h(E_2)\bigr) \geq \delta \diam(E_1') \\ \geq \delta \frac{r^{i_2-i_1}}{C} \diam h(E_1) \geq \delta\frac{r^{2j}}{C} \diam h(E_1).\qedhere
\end{multline*}
\end{proof}

\section{Finitely ramified Julia sets}
\label{sec:selfcoveringfractal}

In this section we show how to apply the theory of quasisymmetries of finitely ramified fractals developed in the last section to the case of finitely ramified Julia sets.  We prove in Section~\ref{subsec:FiniteInvariantBranchCuts} that any connected Julia set for a hyperbolic rational function with a finite invariant branch cut has a natural finitely ramified cell structure, and in Section~\ref{subsec:FindingBranchCuts} we show that any connected Julia set for a hyperbolic polynomial has such a branch cut.  In Section~\ref{subsec:PiecewiseCanonical} we prove that a large class of piecewise canonical homeomorphisms for such Julia sets are quasisymmetries, and finally in Section~\ref{subsec:bubblebath} we apply this theory to show that the bubble bath Julia set (i.e.\ the Julia set for $1-z^{-2}$) has infinitely many quasisymmetries. 
\subsection{Finite invariant branch cuts}\label{subsec:FiniteInvariantBranchCuts}

The goal of this section is to prove the following theorem.

{
\renewcommand{\thetheorem}{\ref{thm:FinitelyRamifiedJulia}}
\begin{theorem}\textFinitelyRamifiedJulia
\end{theorem}
\addtocounter{theorem}{-1}
}

To prove this theorem, we need the following basic lemma about covering maps, which says roughly that the inverse of a finite-sheeted covering map is uniformly continuous when viewed as a multivalued function.

\begin{lemma}\label{lem:InverseCoversUniformlyContinuous}
Let $f\colon X\to Y$ be a covering map between compact, connected, locally connected metric spaces, and let $\epsilon>0$.  Then there exists a $\delta>0$ such that for every connected set $C\subset Y$ of diameter less than~$\delta$, each component of $f^{-1}(C)$ has diameter less than $\epsilon$ and maps homeomorphically to~$C$ under~$f$.
\end{lemma}
\begin{proof}
Since $X$ and $Y$ are both compact, the covering map $f$ must have some finite degree $d\geq 1$. Since $Y$ is compact and locally connected we can choose an open cover $\{U_1,\ldots,U_n\}$ of $Y$ by connected open sets that are evenly covered by~$f$. For each~$i$, let $U_{i,1}',\ldots,U_{i,d}'$ be the components of  $f^{-1}(U_i)$, and note that $f$ maps each $U_{i,j}'$ homeomorphically to~$U_i$.  Let $g_{i,j}\colon U_i\to U_{i,j}'$ denote the inverse of this homeomorphism, and choose an open cover $W_1,\ldots,W_n$ of $Y$ so that $\overline{W_i}\subseteq U_i$ for each~$i$. By compactness, each of the finitely many maps $g_{i,j}$ is uniformly continuous on $\overline{W_i}$, so there exists a single $\delta_1>0$ such that each $g_{i,j}$ maps subsets of $\overline{W_i}$ of diameter less than $\delta_1$ to subsets $U_{i,j}'$ of diameter less than~$\epsilon$.  Let $\delta_2>0$ so that every subset of~$Y$ of diameter less than $\delta_2$ is contained in one of the~$W_i$.  Then $\delta=\min(\delta_1,\delta_2)$ suffices.
\end{proof}

\begin{proof}[Proof of Theorem~\ref{thm:FinitelyRamifiedJulia}]
Suppose $J_f$ has a finite invariant branch cut~$S$. Let $V_n=f^{-n}(S)$ for each $n\geq 1$, and note that $S\subseteq V_1 \subseteq V_2\subseteq \cdots$ since $S$ is invariant.  Define the $n$-cells of $J_f$ to be the closures of the components of~$J_f\setminus V_n$.  Since $S$ is finite and invariant, each $V_n$ is a finite set of periodic and preperiodic points, so by Theorem~\ref{thm:FinitelyManyComponents} there are finitely many $n$-cells for each~$n$. We must show that these cells satisfy the conditions in Definition~\ref{def:finitelyramifiedcellstructure}. Clearly each $n$-cell is compact and connected, and Axioms (1) through (3) hold.    To prove Axiom~(4), it suffices to show that there exists a metric $d$ on $J_f$ and constant $\kappa>0$ so that each component of $J_f\setminus V_n$ has diameter less than $\kappa\lambda^{-n}$ with respect to~$d$.

Since $f$ is hyperbolic, there exists a metric $d$ on $J_f$ with respect to which $f$ is expanding. That is, there exist constants $\epsilon >0$ and $\lambda>1$ so that $d\bigl(f(p),f(q)\bigr) \geq \lambda\,d(p,q)$ for all points $p,q\in J_f$ with $d(p,q)<\epsilon$. By Lemma~\ref{lem:InverseCoversUniformlyContinuous}, there exists a $\delta>0$ so that for any connected set~$C\subseteq J_f$ of diameter less than~$\delta$, each component of $f^{-1}(C)$ has diameter less than $\epsilon$ and maps homeomorphically to~$C$.  In this case, it follows that $f^{-1}(C)$ has diameter at most $\lambda^{-1}\diam(C)$.  Indeed, for every $n\geq 1$ each component of $f^{-n}(C)$ has diameter at most $\lambda^{-n}\diam(C)$ and maps homeomorphically to $C$ under~$f^n$.

Let $S=\{s_1,\ldots,s_k\}$.  Since $f$ is hyperbolic and $J_f$ is connected, $J_f$ is locally connected \cite[Theorem~19.2]{Milnor}. It is also compact, so we can choose a finite cover $\{A_1,\ldots,A_k,B_1,\ldots,B_m\}$ of $J_f$ by connected sets of diameter less than $\delta$ so that each $s_i$ is contained only in $A_i$ and the sets $A_1,\ldots,A_k$ are pairwise disjoint.  We will prove that each $n$-cell in $J_f$ has diameter at most $(2m+1)\delta\lambda^{-n}$.

Let $E$ be an $n$-cell of $X$, let $A_1',\ldots,A_K'$ be the components of the $f^{-n}(A_i)$'s that intersect $E$, and let $B_1',\ldots,B_M'$ be the components of the $f^{-n}(B_i)$'s that intersect~$E$.  Since $S$ is branch cut, we know that $f^n$ maps the interior of $E$ homeomorphically to some component of $J_f\setminus S$.  In particular, for each $i$ at most one component of $f^{-n}(B_i)$ is contained in $E$, and therefore $M\leq m$.  Let $x$ and $y$ be points in $E$.  Since $E$ is connected, there exists a sequence $C_1,\ldots,C_j$ of distinct elements of  $\{A_1',\ldots,A_K',B_1',\ldots,B_M'\}$ such that $x\in C_1$, $y\in C_j$, and $C_i\cap C_{i+1}\ne \emptyset$ for each~$i$.  Since the sets $A_1',\ldots,A_K'$ are pairwise disjoint, no two consecutive $C_i$'s lie in $\{A_1',\ldots,A_K'\}$.  It follows that $j\leq 2M+1\leq 2m+1$.  But each $A_i'$ or $B_i'$ has diameter at most $\lambda^{-n}\delta$, and therefore $d(x,y)\leq (2m+1)\lambda^{-n}\delta$.  We conclude that each $n$-cell of $X$ has diameter at most $(2m+1)\lambda^{-n}\delta$, so our $n$-cells satisfy Axiom~(4) in the definition of a finitely ramified cell structure.
\end{proof}

\begin{remark}
Though we have stated and proven Theorem~\ref{thm:FinitelyRamifiedJulia} in the context of Julia sets for hyperbolic rational maps, the same proof applies to any compact, connected, locally connected metric space $X$ with an expanding self-covering $f\colon X\to X$ and a finite invariant branch cut $S$ for which $X\setminus S$ has finitely many connected components.
\end{remark}

\begin{remark}
One might hope to weaken the hypothesis of Theorem~\ref{thm:FinitelyRamifiedJulia} by removing the requirement that the branch cut be invariant.  Specifically, given any finite branch cut $S\subset J_f$, one can put a finitely ramified cell structure on $J_f$ by defining the $n$-cells be the closures of the complementary components of the set $\bigcup_{i=1}^n f^{-i}(S)$.  Unfortunately, the resulting cell structures are not particularly useful since the diameters of the cells can vary wildly.  For example, if $f(z)=z^2$, then $J_f$ is the unit circle and any one-point set $\{p\}\subset J_f$ is a branch cut. However, there exist points $p$ on the unit circle such that $0<|p-p^{2^n}|<1/n!$ for infinitely many values of~$n$.  In this case, the resulting cell structure has $n$-cells of diameter less than $1/n!$ for infinitely many values of~$n$. So the restriction to the circle of the Euclidean metric is not undistorted with respect to this cell structure since it fails the exponential decay condition of Definition \ref{def:undistorted}.  Because of this pathology, it is not clear how branch cuts which are not invariant can be used to construct quasisymmetries of Julia sets.
\end{remark}

\subsection{Finding invariant branch cuts}\label{subsec:FindingBranchCuts}

In this section we prove several propositions for showing that a given Julia set has a finite invariant branch cut, and use them to prove that any connected Julia set for a hyperbolic polynomial has such a branch cut.

\begin{proposition}\label{prop:MakingBranchCuts}
Let $f\colon\Chat\to\Chat$ be a hyperbolic rational map with connected Julia set~$J_f$, and let $S\subset J_f$ be a finite set satisfying $f(S)\subseteq S$.  If there exists a finite connected graph\/ $\Gamma\subset\Chat$ that contains the critical values of $f$ such that\/ $\Gamma\cap J_f\subseteq S$, then $S$ is a finite invariant branch cut for~$J_f$.
\end{proposition}
\begin{proof}
Let $\Gamma$ be such a graph.  Then $\Chat\setminus \Gamma$ is a disjoint union of finitely many open topological disks.  By the Riemann--Hurwitz formula \cite[\S5.4]{Beardon}, the preimage of each such disk $U$ is again a finite disjoint union of open topological disks, and since the critical values of $f$ are contained in $f^{-1}(\Gamma)$, each of these disks maps homeomorphically to~$U$.  

Now, if $C$ is a connected component of $J_f\setminus S$, then $C$ must be contained in some component $U$ of~$\Chat\setminus \Gamma$. Then $f^{-1}(C)$ has one component $C'$ in each component $U'$ of $f^{-1}(U)$, and each such $C'$ maps homeomorphically to $C$ since $U'$ maps homeomorphically to~$U$.  Thus each connected component of $J_f\setminus S$ is evenly covered by~$f$, so $S$ is a branch cut.
\end{proof}

Our present aim is to apply Proposition~\ref{prop:MakingBranchCuts} to Julia sets of polynomials, but it can also be useful for showing that the Julia sets for non-polynomial rational functions are finitely ramified.  For example, we will use Proposition~\ref{prop:MakingBranchCuts} in Section~\ref{subsec:bubblebath} to show that the bubble bath Julia set has a finitely ramified cell structure.

\subsubsection*{External rays}\label{ref:ExternalRays}
For the next lemma we need to recall the definition of rays, which we will use again in Section~\ref{sec:JuliaSetsInfinitelyManyQuasi}.  Recall that if $f\colon \Chat\to\Chat$ is a rational map with connected Julia set $J_f$, then every Fatou component $U$ for~$f$ must be homeomorphic to an open disk. If we choose a center point $c\in U$, then there exists a Riemann map $\Phi\colon \D\to U$ such that $\Phi(0)=c$, where $\D$ is the open unit disk.  In this case, the images under $\Phi$ of the radial lines in $\D$ are called the \newword{rays} from $c$.

If $f$ is hyperbolic, then $J_f$ is locally connected \cite[Lemma~19.2]{Milnor}, and indeed the boundary $\partial U$ of any Fatou component is locally connected~\cite[Lemma~19.3]{Milnor}.  In this case, by a theorem of Carath\'{e}odory, the map $\Phi$ extends to a continuous surjection $\Phi\colon \overline{\D}\to \overline{U}$, where $\partial\D$ maps onto $\partial U$ \cite[Theorem~17.14]{Milnor}.  If $r$ is a radial line in $\D$ from $0$ to a point $p\in\partial\D$, then the corresponding ray $\Phi(r)$ is said to \newword{land} at the point~$\Phi(p)\in \partial U$.

In the case where $f$ is a polynomial with connected Julia set and $U$ is the Fatou component that contains~$\infty$, the rays from $\infty$ are known as \newword{external rays}.  If $f$ is hyperbolic, then each external ray lands at a point of $J_f$, and every point of $J_f$ is a landing point for at least one external ray~\cite[Theorem~18.3]{Milnor}

%\subsubsection*{External rays}\label{ref:ExternalRays}
%For the next lemma we need to recall the definition of external rays, which we will use again in Section~\ref{sec:JuliaSetsInfinitelyManyQuasi}. If $f(z)$ is a polynomial of degree $d\geq 2$ with connected Julia set and $U_\infty$ is the Fatou component for $f$ in $\Chat$ that contains~$\infty$, there exists a Riemann map $\Phi\colon U_\infty \to \D$  that makes the following diagram commute:
%\[
%\xymatrix{(U_\infty,\infty)\ar[r]^f \ar[d]_\Phi & (U_\infty,\infty)\ar[d]^\Phi \\
%(\D,0)\ar[r]_{z\mapsto z^d} & (\D,0)
%}
%\]
%The map $\Phi$ is called the \newword{B\"ottcher map} for $f$, and is unique up to rotation by~$2\pi/(d-1)$. For any angle $\theta$, the image of the map $\gamma_\theta\colon(0,1)\to U_\infty$ defined by $\gamma_\theta(t)= \Phi^{-1}(te^{t\theta})$ is called an \newword{external ray} for~$f$ with an angle of~$\theta$ (see \cite[\S 18]{Milnor}).  Such a ray \newword{lands} if it can be extended continuously to a map $\gamma_\theta\colon (0,1]\to\C$, in which case $\gamma_\theta(1)$ lies on the Julia set for $f$ and is the \newword{landing point} of the ray.  In the case where the Julia set is locally connected (e.g.~if $f$ is hyperbolic), it is known that every external ray lands, and every point in the Julia set is the landing point of at least one external ray. \jblue{Add a reference for external rays.}

\subsubsection*{Branch cuts for polynomial Julia sets}
The following lemma will be our main tool for finding finite invariant branch cuts for polynomial Julia sets. 

\begin{lemma}\label{lem:PolynomialBranchCuts}
Let $f\colon \C\to\C$ be a hyperbolic polynomial with connected Julia set~$J_f$, and let $S\subset J_f$ be a finite set satisfying $f(S)\subseteq S$.  Then $S$ is a branch cut for $J_f$ if and only if it intersects the boundary of each Fatou component for $f$ that contains a critical value.
\end{lemma}
\begin{proof}
Note that $S$ always intersects the boundary of the unbounded Fatou component for~$f$, since this boundary is the entire Julia set, so we need only worry about the bounded Fatou components for $f$ that contain critical values.  Since $f$ is a hyperbolic polynomial, it follows from \cite[Theorem~133.3]{Steinmetz} that each bounded Fatou component of $f$ is a Jordan domain.  Any bounded Fatou component $U$ that contains a critical value is the image of some bounded Fatou component $V$ that contains a critical point, and the restriction $f\colon \partial V\to\partial U$ is a covering map of Jordan curves of degree two or greater. In particular, no open set containing $\partial U$ is evenly covered by $f$, so any branch cut for $J_f$ must intersect~$\partial U$.

For the converse, suppose that $S$ intersects the boundary of each bounded Fatou component that contains a critical value.  Let $p_1,\ldots,p_m$ be the critical values of $f$ in $\C$, which may or may not be contained in distinct Fatou components.  For each~$i$, let $s_i$ be a point of $S$ on the boundary of the Fatou component that contains~$p_i$, and let $\alpha_i$ be an arc connecting $p_i$ and $s_i$ that intersects the Julia set only at~$s_i$.  Note that if some Fatou component $U$ contains more than one critical value, we can arrange for the corresponding $\alpha_i$'s to not intersect in~$U$.  Let $\Gamma$ be a finite graph in $\Chat$  consisting of the union of the $\alpha_i$'s together with one external ray landing at each~$s_i$, plus the point at~$\infty$.  Then $\Gamma$ contains all the critical values of $f$ and $\Gamma\cap J_f\subseteq S$, so by Proposition~\ref{prop:MakingBranchCuts} we conclude that $S$ is a branch cut for the restriction $f\colon J_f\to J_f$.
\end{proof}

\begin{remark}
In some cases, it is helpful to have a finite invariant branch cut $S$ with the special property that closure of each component of $f^{-1}(S)$ (i.e.~each 1-cell) maps homeomorphically under~$f$.  In the polynomial case, it is not hard to show that this occurs if and only if $S$ includes at least two points from the boundary of each Fatou component that contains a critical value.
\end{remark}

The following proposition is a precise version of the well-known principle that connected Julia sets for hyperbolic polynomials are finitely ramified.

\begin{proposition}\label{prop:PolynomialJuliaSetsFinitelyRamified}Any connected Julia set for a hyperbolic polynomial has a finite invariant branch cut.
\end{proposition}
\begin{proof}
Let $f$ be a hyperbolic polynomial with connected Julia set~$J_f$.  We claim first that each bounded Fatou component for $f$ has at least one periodic or pre-periodic point on its boundary.

By \cite[Theorem~133.3]{Steinmetz}, each bounded Fatou component of $f$ is a Jordan domain.  Since $f$ is hyperbolic, every orbit in the Fatou set converges to an attracting periodic orbit \cite[Theorem~19.1]{Milnor}, so every periodic Fatou component contains a point from an attracting cycle.  By \cite[Theorem~8.6]{Milnor}, it follows that every periodic cycle of bounded Fatou components contains a critical point.  In particular, if $U$ is a Fatou component of period $n$, then $f^n$ maps the Jordan curve $\partial U$ to itself by a covering map of degree two or greater, and hence $f^n$ has a fixed point on~$\partial U$.  Thus every periodic Fatou component has at least one periodic point on its boundary.  By Sullivan's nonwandering domain theorem \mbox{\cite[Theorem~16.4]{Milnor}}, every Fatou component is either periodic or pre-periodic, and therefore every bounded Fatou component has at least one periodic or pre-periodic point on its boundary.

Let $U_1,\ldots,U_n$ be the bounded Fatou components of $f$ that contain critical points.   By the argument in the previous paragraph, each $f(U_i)$ has at least one periodic or pre-periodic point on its boundary.  In particular, there exists a finite invariant set $S\subset J_f$ which intersects the boundary of each~$f(U_i)$.  By Lemma~\ref{lem:PolynomialBranchCuts}, it follows that $S$ is a branch cut for~$J_f$, so $J_f$ is finitely ramified.
\end{proof}

Note that there is no analog of  Proposition~\ref{prop:PolynomialJuliaSetsFinitelyRamified} for connected Julia sets of hyperbolic rational functions, since these are not always finitely ramified.  For example, there exists a hyperbolic rational function whose Julia set is homeomorphic to a Sierpi\'{n}ski carpet~\cite[Appendix~F]{MilnorQuadraticRational}.  Such a Julia set cannot be finitely ramified, since the complement of finitely many points in a Sierpi\'{n}ski carpet is always connected.

\subsection{Piecewise canonical homeomorphisms}\label{subsec:PiecewiseCanonical}

In this section we prove Theorems \ref{thm:JuliaSetUndistorted} and \ref{thm:PiecewiseCanonicalQuasisymmetries}, which state that the Euclidean metric on a finitely ramified Julia set is undistorted, and that therefore piecewise canonical homeomorphisms of such a Julia set are quasisymmetries.  These will be our main tools for proving that Julia sets have infinitely many quasisymmetries in Section~\ref{subsec:bubblebath} and in Section~\ref{sec:JuliaSetsInfinitelyManyQuasi}.

{
\renewcommand{\thetheorem}{\ref{thm:JuliaSetUndistorted}}
\begin{theorem}\textJuliaSetUndistorted
\end{theorem}
\addtocounter{theorem}{-1}
}
\begin{proof}
Note first that, since the Julia set for a hyperbolic rational map cannot be the entire Riemann sphere, we can conjugate $f$ by a M\"{o}bius transformation to ensure that $\infty\notin J_f$.  In this case, the restrictions to $J_f$ of the spherical metric on $\Chat$ and the Euclidean metric on $\mathbb{C}$ are quasisymmetrically equivalent (since they are bilipschitz equivalent on any compact subset of~$\C$), so it suffices to prove that the theorem holds for the restriction of the Euclidean metric.

We will prove that the restriction of the Euclidean metric to $J_f$ satisfies the hypotheses of Proposition~\ref{prop:undistorted}.  To define the collection $\M$, let $\delta>0$ so that $f$ is injective on every subset of $J_f$ of diameter less than or equal to~$\delta$.  Let $P_f$ be the postcritical set of $f$, i.e.\ the union of the forward orbits of the critical points.  Since $f$ is hyperbolic, the Julia set $J_f$ is disjoint from $P_f$. Let $\delta'=\min(\delta,d(J_f,P_f))$.  Since the intersection of any nested collection of cells is a single point, there exists an $N\in \N$ so that all $N$-cells in $J_f$ have diameter strictly less than $\delta'/2$.  We claim that the collection $\mathcal{M}$ of all cells of level $N$ or less suffices.

For each $M$ which is either a cell of $\M$ or a union of two neighboring cells of $\M$, let $U_M$ be an open disk centered at some point of $M$ of radius~$\delta'$.  Since $M$ is either a cell or pair of neighboring cells, we know that $\diam(M)<\delta'$, so $M \subseteq U_{M}$.  Note also that $U_{M}$ is disjoint from~$P_f$.  By V\"ais\"al\"a's egg yolk principle \cite[Theorem~2.7]{Vai1}, there exists a homeomorphism $\eta_{M}\colon [0,\infty)\to [0,\infty)$ so that $g|_{M}$ is an $\eta_{M}$-quasisymmetric embedding for every conformal map $g\colon U_{M}\to \mathbb{C}$.  Let $\eta_1\colon [0,\infty)\to [0,\infty)$ be the maximum of $\eta_{M}$ as $M$ ranges over all cells and neighboring pairs of cells in $\M$.

Now let $E$ be an $n$-cell or union of two neighboring $n$-cells for some $n>N$. Then $f^{n-N}$ maps $E$ to some set $M$ by a cellular homeomorphism, where $M$ is either an $N$-cell or the union of a neighboring pair of $N$-cells.  Since $U_M$ is disjoint from $P_f$, the inverse map $M\to E$ extends to a conformal map $U_M\to \mathbb{C}$ which is a branch of the inverse of $f^{n-N}$.  In particular, the inverse map $M\to E$ is an $\eta_1$-quasisymmetry.  It follows that $f^{n-N} \colon E\to M$ is an $\eta_2$-quasisymmetry, where $\eta_2\colon [0,\infty)\to[0,\infty)$ is the homeomorphism defined by $\eta_2(0)=0$ and $\eta_2(t)=1/\eta_1^{-1}(1/t)$ for $t>0$.

Now, each cell or pair of neighboring cells in $\M$ maps to itself by the identity map which is an $\eta_0$-quasisymmetry, where $\eta_0$ is the identity homeomorphism.  We conclude that $\mathcal{M}$ and $\eta=\max(\eta_0,\eta_2)$ satisfy the hypotheses of Proposition~\ref{prop:undistorted}, so the restriction of the Euclidean metric is undistorted with respect to the induced finitely ramified cell structure.
\end{proof}

{
\renewcommand{\thetheorem}{\ref{thm:PiecewiseCanonicalQuasisymmetries}}
\begin{theorem}\textPiecewiseCanonicalQuasisymmetries
\end{theorem}
\addtocounter{theorem}{-1}
}
\begin{proof}
Let $S$ be a finite invariant branch cut for $J_f$, and let $B$ be the set of breakpoints of $h$.  Since points of $B$ are periodic or pre-periodic, the union $B'=\bigcup_{n\geq 1} f^n(B)$ of the forward orbits of the points in $B$ is finite.  Replacing $S$ by $S\cup B'$ if necessary, we may assume that $B\subseteq f^{-1}(S)$. 
 We will prove that $h$ is piecewise cellular with respect to this cell structure.

Let $E_1,\ldots,E_k$ be the $1$-cells of $J_f$, i.e.\ the closures of the connected components of $J_f\setminus f^{-1}(S)$.  Since $B\subseteq f^{-1}(S)$, the map $h$ restricts to a canonical homeomorphism on the interior $\mathrm{int}(E_i)$ of each~$E_i$.  In particular, there exist $m_i,n_i\geq 0$ so that $f^{m_i}\circ h$ agrees with $f^{n_i}$ on~$\mathrm{int}(E_i)$.  Let $M=1+\max(n_1,\ldots,n_k)$.  We claim that $h$ restricts to a cellular map on each $M$-cell of~$J_f$.

Let $E$ be an $M$-cell of $J_f$.  Then $E\subseteq E_i$ for some~$i$, so $f^{m_i}\circ h$ agrees with $f^{n_i}$ on~$E$.  Since $S$ is a branch cut, $f$ is injective on the interior of each $1$-cell, and indeed on the interior of each $k$-cell for $k\geq 1$.  Since $n_i\leq M-1$, it follows that $f^{n_i}$ maps $\mathrm{int}(E)$ homeomorphically to the interior of an $(M-n_i)$-cell~$E_*$. Then $f^{m_i}\circ h$ maps $\mathrm{int}(E)$ homeomorphically to $\mathrm{int}(E_*)$, so $f^{m_i}$ maps the interior of $h(E)$ homeomorphically to~$\mathrm{int}(E_*)$.  But each connected component of $f^{-m_i}(\mathrm{int}(E_*))$ is the interior of an $(M+m_i-n_i)$-cell that maps homeomorphically to~$\mathrm{int}(E_*)$ under~$f^{m_i}$.  Then the interior of $h(E)$ must be the interior of one of these cells, and since $h(E)$ is the closure of its interior it follows that $h(E)$ is an $(M+m_i-n_i)$-cell.

Now, if $E'\subseteq E$ is an $(M+N)$-cell for some $N\geq 0$, then the same argument shows that $h(E')$ is some $(M+N+m_i-n_i)$-cell contained in $h(E)$.  It follows that $h$ is cellular on~$E$, and therefore $h$ is piecewise cellular.  Since the metric on $J_f$ is undistorted by Theorem~\ref{thm:JuliaSetUndistorted}, it follows from Theorem~\ref{thm:PiecewiseCellular} that $h$ is a quasisymmetry. \end{proof}

\begin{remark}
Given a rational map $f$ as in Theorem \ref{thm:PiecewiseCanonicalQuasisymmetries}, let $C$ be any collection of periodic and pre-periodic points so that $f(C) = C$, that is suppose $C$ is a union of grand orbits of $f$.  Then, the piecewise canonical homeomorphisms of $J_f$ with breakpoints contained in $C$ forms a group.

For example, consider the basilica Julia set defined in Example~\ref{ex:Basilica}.  Let $p$ be the fixed point $(1-\sqrt{5})/2$, and let $C=\bigcup_{n\geq 0} f^{-n}(p)$ be the set of all ``pinch points'' in the Julia set.  In~\cite{BeFo1}, the authors investigated \newword{basilica Thompson group} $T_B$ of all piecewise canonical homeomorphisms of $J_f$ whose allowed breakpoints lie in~$C$. This group is finitely generated, and Lyubich and Merenkov showed that it is dense in the planar quasisymmetry group of $J_f$ in a certain sense~\cite{LyMe}.

More generally, given any rational map $f$ as in Theorem \ref{thm:PiecewiseCanonicalQuasisymmetries} and any finite invariant branch cut $S$, let $C=\bigcup_{n\geq 0} f^{-n}(S)$.  Then the group of all piecewise canonical homeomorphisms of $J_f$ with breakpoints in $C$ is isomorphic to a rearrangement group as defined in~\cite{BeFo2}.  This rearrangement group uses a colored replacement system, whose replacement rules are determined by how each 1-cell subdivides into 2-cells.  Note that if any of the cells of $J_f$ have more than two boundary points then it is necessary to use hypergraphs instead of graphs.
\end{remark}

\subsection{A rational example}
\label{subsec:bubblebath}

\begin{figure}
\centering
$\underset{\textstyle\text{(a)\rule{0pt}{12pt}}}{\includegraphics{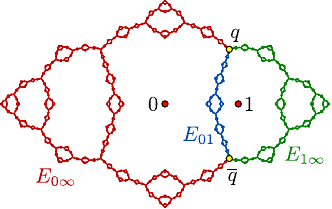}}\qquad\underset{\textstyle\text{(b)\rule{0pt}{12pt}}}{\includegraphics{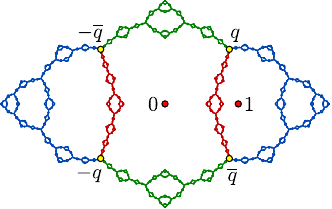}}$
\caption{(a) The Julia set $J_f$ for $f(z)=1-z^{-2}$.  The fixed points $q$ and $\overline{q}$ form an invariant branch cut. (b) The 1-cells of the corresponding finitely ramified cell structure.}
\label{fig:BubbleBath}
\end{figure}
In this section we consider the ``bubble bath'' Julia set for the rational map $f(z)=1-z^{-2}$, which is shown in Figure~\ref{fig:BubbleBath} and discussed in Example \ref{ex:BubbleBath}.   Our main theorem is the following.  

{
\renewcommand{\thetheorem}{\ref{thm:BubbleBathModularGroup}}
\begin{theorem}\textBubbleBathModularGroup
\end{theorem}
\addtocounter{theorem}{-1}
}

 The existence of a copy of $\Z_2*\Z_3$ in the group of homeomorphisms of $J_f$ was first observed by Jasper Weinburd~\cite{Wei}.

The following proposition lists some of the basic properties of $f$ and~$J_f$.

\begin{lemma}\label{lem:PropertiesBubbleBath}
Let $f(z)=1-z^{-2}$. Then:
\begin{enumerate}
    \item The critical points for $f$ are $0$ and $\infty$, both of which lie in a $3$-cycle with the point\/~$1$. In particular, $f$ is hyperbolic.\smallskip
    \item The Julia set $J_f$ is connected.\smallskip
    \item The points $0$, $1$, and $\infty$ lie in distinct Fatou components $U_0$, $U_1$, and $U_\infty$, each of which is a Jordan domain.\smallskip
    \item The map $f$ has one real fixed point $p\approx -0.7549$ and two complex conjugate fixed points
    \[
    q\approx 0.8774+0.7449i\qquad\text{and} \qquad \overline{q} \approx 0.8774 - 0.7449i.
    \]
    Both $q$ and $\overline{q}$ lie on $\partial U_0\cap \partial U_1 \cap \partial U_\infty$, and $J_f\setminus\{q,\overline{q}\}$ has exactly three components $E_{01},E_{0\infty},E_{1\infty}$ (see Figure~\ref{fig:BubbleBath}a), where each $E_{ij}$ intersects both $\partial U_i$ and $\partial U_j$ but not the remaining $\partial U_{k}$.
\end{enumerate}
\end{lemma}
\begin{proof}
Statement (1) is easy to verify.  For statement~(2), Milnor proves in \cite[Lemma~8.2]{MilnorQuadraticRational} that the Julia set for a quadratic rational map is connected as long as the orbits of the critical points do not converge to fixed points.

For statement~(3), observe that $0$, $\infty$, and $1$ are superattracting fixed points for~$f^3$.  By the Sullivan classification of Fatou components~\cite[Theorem~16.1]{Milnor}, it follows that $U_0$, $U_\infty$, and $U_1$ are the immediate basins for $0$, $\infty$, and $1$ with respect to~$f^3$, and therefore these sets are all distinct.  Furthermore, since $f$ is postcritically finite, has exactly two critical points, and is not conjugate to a polynomial, it follows from a result of Pilgrim~\cite{Pil} that every Fatou component for $f$ is a Jordan domain.

For statement~(4), the fixed points of $f$ are the roots of the polynomial $z^3-z^2+1$, and it is easy to check that one is real and two are complex conjugates.  To prove that $q,\overline{q}\in\partial U_0\cap \partial U_1\cap\partial U_\infty$, observe that $f^3$ maps each of the Jordan curves $\partial U_0$, $\partial U_1$, and $\partial U_\infty$ to itself with degree four.  Since any degree four map from a circle to itself has at least three fixed points, it follows that $f^3$ must have at least three fixed points on each of $\partial U_0$, $\partial U_1$, and $\partial U_\infty$.  It is easy to check that the fixed points for~$f^3$ are precisely the points $\{0,1,\infty,p,q,\overline{q},r_1,r_2,r_3\}$, where
\[
r_1\approx -2.2470,\qquad r_2\approx 0.8019,\qquad r_3\approx -0.5550
\]
are the three real roots of the polynomial $z^3+2z^2-z-1$ (which form a $3$-cycle under~$f$). Of these, only $\{p,q,\overline{q},r_1,r_2,r_3\}$ lie on the Julia set, and all of these are real except $q$ and $\overline{q}$. 
 Since $U_0$, $U_1$, and $U_\infty$ map to themselves under complex conjugation, each of the curves $\partial U_0,\partial U_1,\partial U_\infty$ contains at most two real points, and therefore all of these curves must contain both $q$ and~$\overline{q}$.

All that remains is to prove that $J_f\setminus\{q,\overline{q}\}$ has exactly three components.  For this we will use the notion of filling of Julia sets developed in Appendix~\ref{ap:RemovingFiniteSets}.  Note first that by \cite[Theorem~8.6]{Milnor}, the cycle $\{0,\infty,1\}$ is the only attracting cycle for~$f$.  Since $f$ is hyperbolic, it follows from \cite[Theorem~19.1]{Milnor} that the orbit of every point in the Fatou set must converge to this cycle, so $U_0$, $U_1$, and $U_\infty$ must be the only periodic Fatou components.  By Lemma~\ref{lem:ComplementaryComponents}, every Fatou component that has $q$ or $\overline{q}$ on its boundary must be periodic, so $U_0$, $U_1$, and $U_\infty$ are the only Fatou components whose boundaries contain $q$ or $\overline{q}$.  It follows that the filling $E$ of $J_f\setminus \{q,\overline{q}\}$ is precisely the complement of $U_0\cup U_1\cup U_\infty\cup \{q,\overline{q}\}$ in~$\Chat$.  By Lemma~\ref{lem:ComponentsFilling}, the components of $J_f\setminus \{q,\overline{q}\}$ are in one-to-one correspondence with the components of~$E$, so it suffices to prove that $E$ has exactly three connected components.

To see this, choose arcs $A_0$, $A_1$, and $A_\infty$ connecting $q$ and $\overline{q}$ whose interiors are contained in $U_0$, $U_1$, and $U_\infty$, respectively.  Then $A_0\cup A_1\cup A_\infty$ separates the plane into three regions $C_{01},C_{0\infty},C_{1\infty}$, where each $C_{ij}$ is bounded by $A_i\cup A_j$, and each component of $E$ must be contained in one of these three regions. For each $C_{ij}$, let $B_{ij}$ and $B_{ji}$ denote the open arcs of $\partial U_i\setminus\{q,\overline{q}\}$ and $\partial U_j\setminus\{q,\overline{q}\}$, respectively, that are contained in $C_{ij}$. Then $\bigcup_{i\ne j} B_{ij}$ is the boundary of $U_0\cup U_1\cup U_\infty$ in $\Chat\setminus \{q,\overline{q}\}$, which is the same as the boundary of $E$ in $\Chat\setminus \{q,\overline{q}\}$. Since $E$ is closed in $\Chat\setminus\{q,\overline{q}\}$, its connected components are also closed in $\Chat\setminus\{q,\overline{q}\}$,  so each connected component of $E$ must intersect and hence contain one of the sets $B_{ij}$. But no $B_{ij}$ separates $\Chat\setminus\{q,\overline{q}\}$ by itself, so any connected component of $E$ in $C_{ij}$ must contain $B_{ij}\cup B_{ji}$, and therefore $E$ has exactly three connected components $E_{01},E_{0\infty},E_{1\infty}$ with $E_{ij}\subseteq C_{ij}$.\end{proof}

\begin{proof}[Proof of Theorem~\ref{thm:BubbleBathModularGroup}]
Let $U_0$, $U_1$, $U_\infty$, $q$, $E_{01}$, $E_{0\infty}$, and $E_{1\infty}$ be as in Lemma~\ref{lem:PropertiesBubbleBath}.  For each $w\in\{0,1,\infty\}$, choose an arc $A_w$ connecting $q$ and $\overline{q}$ whose interior lies in $U_w$ and which goes through~$w$.  Then $\Gamma=A_0\cup A_1\cup A_\infty$ is a finite graph in~$\Chat$ that contains the critical values $1$ and $\infty$ and intersects $J_f$ only at $\{q,\overline{q}\}$, so by Proposition~\ref{prop:MakingBranchCuts} the set $S=\{q,\overline{q}\}$ is a finite invariant branch cut for~$J_f$.  

To prove that the quasisymmetry group for $J_f$ contains $\Z_2*\Z_3$, we use the ping-pong lemma for free groups stated in Appendix~\ref{ap:PingPongLemmas}. To construct the desired homeomorphisms $h$ and $k$, observe first that $f^{-1}(S)=\{q,\overline{q},-q,-\overline{q}\}$.  Since $f^{-1}(\partial U_\infty)=\partial U_0$ and $f^{-1}(\partial U_1)=\partial U_\infty$, both $-q$ and $-\overline{q}$ are contained in $\partial U_0\cap \partial U_\infty$ and hence in $E_{0\infty}$.  Then $E_{01}$ and $E_{1\infty}$ must be $1$-cells in~$J_f$, so $f$ maps $E_{1\infty}$ homeomorphically to $E_{01}$ and $E_{01}$ homeomorphically to $E_{0\infty}$, as indicated in Figure~\ref{fig:BubbleBath}(b).

\begin{figure}
\centering
$\raisebox{-0.47\height}{\includegraphics{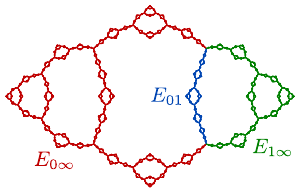}}\quad\longrightarrow\quad\raisebox{-0.47\height}{\includegraphics{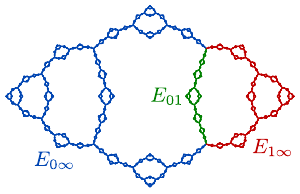}}$
\caption{An order $3$ quasisymmetry $k$ of the bubble bath Julia set. Here $E_{01}\cup E_{1\infty}$ maps to $E_{0\infty}\cup E_{01}$ via $f$ and $E_{0\infty}$ maps to $E_{1\infty}$ via a branch of~$f^{-2}$.}
\label{fig:BubbleBathQuasisymmetry}
\end{figure}
Let $h\colon J_f\to J_f$ be the isometry $h(z)=-z$, and let $k\colon J_f\to J_f$ be the homeomorphism that maps $E_{1\infty}\cup E_{01}$ to $E_{01}\cup E_{0\infty}$ via $f$ and maps $E_{0\infty}$ to $E_{1\infty}$ by the inverse of the homeomorphism $f^2\colon E_{1\infty}\to E_{0\infty}$, as shown in Figure~\ref{fig:BubbleBathQuasisymmetry}. Note that $k$ is piecewise canonical with breakpoints at $q$ and $\overline{q}$, so by Theorem~\ref{thm:PiecewiseCanonicalQuasisymmetries} it is a quasisymmetry.  Note also that $h$ has order~$2$ and $k$ has order~$3$.  We will use the ping-pong lemma to prove that $\langle h,k\rangle\cong \Z_2*\Z_3$.

Let $X_K=E_{1\infty}\cup E_{01}$, and let $X_H=E_{0\infty}$.  Then $k(X_H)=E_{1\infty}$ and $k^2(X_H)=E_{01}$, both of which are proper subsets of $X_K$.  Furthermore, $h(E_{1\infty})$ and $h(E_{01})$ are connected sets with boundary $\{-q,-\overline{q}\}$ and are therefore 1-cells contained in~$E_{0\infty}$.  But $E_{0\infty}$ is the union of the four remaining 1-cells (as shown in Figure~\ref{fig:BubbleBath}b), and therefore $h(X_K)$ is a proper subset of~$X_H$.  By the ping-pong lemma, we conclude that $\langle h,k\rangle\cong \Z_2*\Z_3$.
\end{proof}

\begin{remark}
The quasisymmetries we produce in the proof of Theorem~\ref{thm:BubbleBathModularGroup} are actually \newword{spherical}, i.e.\ they extend to orientation-preserving homeomorphisms of the Riemann sphere.
\end{remark}

\begin{remark}
The group of quasisymmetries of this Julia set is actually much larger than this $\Z_2*\Z_3$.  For example, it is not hard to show that $J_f$ has uncountably many cellular homeomorphisms, all of which are quasisymmetries.  Weinburd shows that the group of piecewise canonical spherical homeomorphisms whose breakpoints lie in the grand orbits of $q$ and $\overline{q}$ is finitely generated, virtually simple, and contains copies of Thompson's group~$T$~\cite{Wei}.
\end{remark}

\section{Polynomial Julia sets with infinitely many quasisymmetries}
\label{sec:JuliaSetsInfinitelyManyQuasi}

The goal of this section is to prove Theorem~\ref{thm:BigTheorem} from the introduction, which asserts that the Julia sets for two large families of hyperbolic polynomials have infinitely many quasisymmetries.

Our proofs will make use of some known results on the geometry of  Julia sets for a postcritically finite polynomial~$f$.  First, recall that the \newword{filled Julia set} $K_f$ is the complement of the Fatou component that contains~$\infty$, i.e.\ the union of the Julia set $J_f$ with all of the bounded Fatou components for~$f$. Since $f$ is postcritically finite, each Fatou component $U\subseteq K_f$ contains a unique point, called the \newword{center} of $U$, whose forward orbit intersects the postcritical set.  The rays in $U$ emanating from this center point are known as \newword{internal rays} (see the discussion of \hyperref[ref:ExternalRays]{external rays} in Section~\ref{subsec:FindingBranchCuts}).  An arc $A\subseteq K_f$ is called a \newword{regulated arc} if the intersection of $A$ with the closure of each bounded Fatou component is contained in the union of at most two internal rays.
\begin{figure}[tb]
\centering
$\underset{\textstyle\rule{0pt}{14pt}\text{(a)}}{\includegraphics{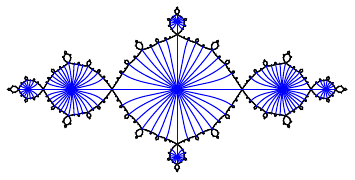}}$\hfill
$\underset{\textstyle\rule{0pt}{14pt}\text{(b)}}{\includegraphics{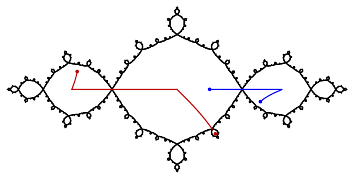}}$
\caption{(a) Some internal rays for the basilica Julia set. (b)~Two regulated arcs in the filled basilica.}
\label{fig:AllowedArcs}
\end{figure}%
For example, Figure~\ref{fig:AllowedArcs}(a) shows some internal rays for the basilica Julia set, and  Figure~\ref{fig:AllowedArcs}(b) shows two regulated arcs for the basilica Julia set.

Douady and Hubbard proved that every pair of points in $K_f$ are connected by a unique regulated arc~\cite{OrsayNotes}. They also defined a \newword{projection function} $\pi_U\colon K_f\to \overline{U}$.  Specifically, if $p\in U$ then $\pi_U(p) = p$, and if $p\in K_f\setminus U$ then $\pi_U(p)$ is the unique point at which the regulated arc from $p$ to $c$ intersects  $\partial U$.   They proved that this function is continuous and is constant on each connected component of $K_f\setminus\overline{U}$ \cite[Proposition~2.5]{OrsayNotes}.  In particular, if $p\in \partial U$, then $\pi_U^{-1}(p)$ is precisely the union of $\{p\}$ with all of the connected components of $K_f\setminus \{p\}$, except for the one that contains~$U$.

\subsection{Unicritical polynomials}\label{subsec:Unicritical}

The goal of this section is to prove the following theorem, which is part (1) of Theorem~\ref{thm:BigTheorem} from the introduction.

\begin{theorem}\label{thm:UnicriticalTheorem}
Let $f$ be a unicritical polynomial of degree $d\geq 2$ whose critical point is periodic.  Then the quasisymmetry group of $J_f$ contains  $\Z_d*\Z_n$ for some~$n\geq 2$.
\end{theorem}

Note that any unicritical polynomial $f$ with periodic critical point has connected Julia set~\cite[Theorem~9.5]{Milnor}. Our proof of Theorem~\ref{thm:UnicriticalTheorem} will use \hyperref[ref:ExternalRays]{external rays} for~$f$.  Recall that these are the images of the radial lines under a Riemann map $\Phi\colon\D\to U$, where $U$ is the Fatou component for $f$ that contains~$\infty$.  B\"{o}ttcher proved that there exists such a $\Phi$ that satisfies $\Phi(z^d) = f(\Phi(z))$ for all $z\in\D$, where $d$ is the degree of~$f$ \cite[Theorems~9.1 and~9.3]{Milnor}.  It follows that $f$ maps external rays to external rays.  An external ray is \newword{periodic} if it maps to itself under some iterate of~$f$.

We will call a fixed point $p$ for $f$ \newword{rotational} if $p\in J_f$ and $p$ is the landing point for some periodic external ray $R$ which is not a fixed ray. For example, the fixed point $p=\bigl(1-\sqrt{5}\bigr)/2$ for the basilica polynomial $f(z)=z^2-1$ is rotational, since it is the landing point of the $1/3$ and $2/3$ rays (see Example~\ref{ex:Basilica}). As shown in Figure~\ref{fig:BasilicaRearrangement} for the basilica, rotational fixed points often have associated finite-order quasisymmetries.

\begin{lemma}\label{lem:ExistsRotationalFixedPoint}
Let $f$ be a unicritical polynomial whose critical point is periodic but not fixed.  Then $f$ has a rotational fixed point.
\end{lemma}
\begin{proof}
Let $d$ be the degree of~$f$.  Since $f$ is postcritically finite, every fixed point for $f$ must be either repelling or superattracting~\cite[Corollary~14.5]{Milnor}. In particular, $f$ has no parabolic fixed points, so the polynomial $f(z)-z$ has no common roots with its derivative $f'(z)-1$.  It follows that the roots of $f(z)-z$ are all distinct, so $f$ has exactly $d$ distinct fixed points. Since the critical point of $f$ is not fixed, all of these fixed points are repelling and lie in~$J_f$. By \cite[Theorem~18.11]{Milnor}, each of these is the landing point of at least one periodic external ray.  But the mapping $\theta\mapsto d\theta$ on the unit circle has exactly $d-1$ fixed points, namely the integer multiples of $1/(d-1)$.  By B\"ottcher's theorem, it follows that $f$ has exactly $d-1$ external rays which are fixed.  We conclude that at least one fixed point $p$ of $f$ must be a landing point for a periodic ray which is not fixed, and therefore $p$ is rotational.
\end{proof}

If $p$ is a rotational fixed point and $R$ is a non-fixed periodic ray that lands at $p$, then each ray in the orbit of $R$ must land at $p$ as well, and these rays subdivide the complex plane into two or more closed regions, which we refer to as \newword{sectors}. 

\begin{lemma}\label{lem:RotationAtFixedPoint}
Let $f$ be a unicritical polynomial with periodic critical point. 
 Suppose $f$ has a rotational fixed point~$p$, and let $S_1,\ldots,S_n$ be a corresponding partition into sectors.  Then there exists a quasisymmetry of $J_f$ of order~$n$ that cyclically permutes the sets $\{J_f\cap S_i\}_{i=1}^n$.\end{lemma}
\begin{proof}
Let $R_1,\ldots, R_n$ ($n\geq 2$) be an orbit of external rays that land at $p$, where $R_{i+1}=f(R_i)$ for each~$i$.  Let $S_i$ denote the sector that lies counterclockwise from~$R_i$, and suppose without loss of generality that $S_n$ contains the critical point of~$f$.  Then $f$ maps each $S_i$ homeomorphically to $S_{i+1}$, so $f$ maps each $J_f\cap S_i$ to $J_f\cap S_{i+1}$.  Moreover, $f^{n-1}$ maps $J_f\cap S_1$ homeomorphically to $J_f\cap S_n$, and this homeomorphism has an inverse $g\colon J_f\cap S_n\to J_f\cap S_1$.  Let $h\colon J_f\to J_f$ be the homeomorphism of order $n$ that maps $J_f\cap S_i$ to $J_f\cap S_{i+1}$ via~$f$ for each $1\leq i\leq n-1$, and maps $J_f\cap S_n$ to $J_f\cap S_1$ via~$g$.  We claim that $h$ is a quasisymmetry.

By Proposition~\ref{prop:PolynomialJuliaSetsFinitelyRamified}, we know that $J_f$ is finitely ramified.  But each connected component $C$ of $J_f\setminus\{p\}$ lies in one of the sectors $S_i$, so $h$ either agrees with either $f$ or $g$ on $C$.  Since $f^{n-1}\circ g$ is the identity, it follows that $h$ is piecewise canonical with a single breakpoint at~$p$, so $h$ is a quasisymmetry by Theorem~\ref{thm:PiecewiseCanonicalQuasisymmetries}.
\end{proof}

\begin{remark}
Lemma~\ref{lem:RotationAtFixedPoint} can be generalized considerably.  First, though the conclusion of the lemma only produces a single quasisymmetry of order~$n$, there is actually a natural copy of the symmetric group $\Sigma_n$ in the quasisymmetry group which permutes the $n$ sectors.  Second, there is no need for $f$ to be unicritical---the lemma works just as well if $f$ has multiple critical points, as long as all of them lie in a single sector~$S_i$.  Indeed, as long as there is any sector $S_j$ that does not contain a critical point, then we can construct a quasisymmetry of order two that swaps $J_f\cap S_j$ with $J_f\cap f(S_j)$.  Finally, note that it is often possible to apply this lemma to points $p$ of period $k\geq 2$ by regarding $p$ as a fixed point of~$f^k$.
\end{remark}

\begin{proof}[Proof of Theorem~\ref{thm:UnicriticalTheorem}]
Let $f$ be a unicritical polynomial of degree $d\geq 2$ whose critical point is periodic.  We will use the ping-pong lemma for free products (see Appendix~\ref{ap:PingPongLemmas}) to produce a subgroup of the quasisymmetry group isomorphic to $\Z_d* \Z_n$. 

First, conjugating $f$ by an affine function, we may assume that $f(z)=z^d+c$ for some complex constant~$c$, with $0$ being the critical point.  If $c=0$, then the Julia set $J_f$ is the unit circle in the complex plane, and it is well known that the quasisymmetry group of a circle contains $\Z_n * \Z_d$ for all $n\geq 2$.

Suppose then that $c\ne 0$.  Then $0$ is not fixed, so by Lemma~\ref{lem:ExistsRotationalFixedPoint} the polynomial $f$ has at least one rotational fixed point~$p$. Let $S_1,\ldots,S_n$ be a corresponding partition into sectors, where $S_1$ is the sector that contains~$0$.  By Lemma~\ref{lem:RotationAtFixedPoint}, there exists a quasisymmetry $h$ of $J_f$ of order~$n$ that cyclically permutes the sets $\{J_f\cap S_i\}_{i=1}^n$, say $h(S_i)=S_{i+1}$ for each $i<n$. Let $k\colon J_f\to J_f$ be the map $k(z)=e^{2\pi i/d}z$, which is a symmetry of $J_f$ since $f\circ k=f$. Note that $k$ is an isometry and hence a quasisymmetry. Our plan is to apply the ping-pong lemma to the cyclic groups $H=\langle h\rangle\cong \Z_n$ and $K=\langle k\rangle\cong \Z_d$.

Let $U$ be the Fatou component for $f$ that contains~$0$, and let $\pi_U\colon K_f\to \overline{U}$ be the projection of $K_f$ onto $U$.  Let $a=\pi_U(p)$ and define
\[
X_K = J_f \cap \pi_U^{-1}(\overline{U}\setminus \{a\}).
\]
Note that $\pi_U^{-1}(\overline{U}\setminus \{a\})$ is connected since $\overline{U}\setminus\{a\}$ is connected and each $\pi_U^{-1}(b)$ for $b\ne a$ is path connected by regulated arcs.  In particular, $\pi_U^{-1}(\overline{U}\setminus\{a\})$ lies entirely within the sector~$S_1$, so $X_K\subset S_1$.  Next, define
\[
X_H = J_f\cap (S_2\cup \cdots \cup S_n).
\]
Note that $X_H$ is disjoint from $X_K$, and $h^i(X_K)\subseteq J_f\cap S_{i+1}\subseteq X_H$ for all $1\leq i<n$.

Finally, since $f\circ k=f$, the grand orbit of $0$ is invariant under $k$. In particular, $k$ maps centers of Fatou components to centers of Fatou components, so it maps internal rays to internal rays and hence maps regulated arcs to regulated arcs.  It follows that $\pi_U\circ k=k\circ \pi_U$.  But $X_H\subseteq \pi_U^{-1}(a)$ since $X_H$ is disjoint from $X_K$, so $k^i(X_H)\subseteq \pi_U^{-1}(k^i(a))$ for all~$i$.  Since the points $k^i(a)$ are distinct for $0\leq i<d$, we conclude that $k^i(X_H)\subseteq X_K$ for all $1\leq i<d$, so $\langle h,k\rangle \cong \Z_n*\Z_d$ by the ping-pong lemma for free products.
\end{proof}

\begin{example}
Consider the unicritical polynomial $f(z)=z^3+c$, where $c\approx 0.2099 + 1.0935i$ is chosen so that the critical point $0$ has period~$4$.  The Julia set for this polynomial is shown in Figure~\ref{fig:FunkyCubic}.
\begin{figure}
\centering
\includegraphics{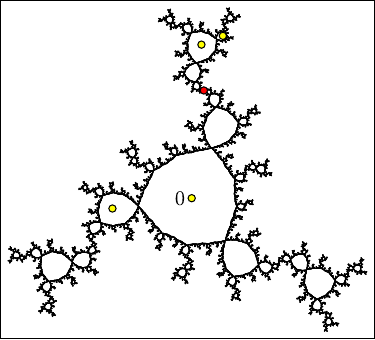}
\caption{Julia set for $f(z)=z^3+c$, where $c\approx 0.2099+1.0935i$. The four points of the critical cycle are shown in yellow, and the rotational fixed point is shown in red.}
\label{fig:FunkyCubic}
\end{figure}%
By Lemma~\ref{lem:ExistsRotationalFixedPoint}, this polynomial must have at least one rotational fixed point, and indeed there is such a point at $p\approx 0.0816+0.7257i$ which is the landing point of exactly two external rays.  Then the quasisymmetry group of $J_f$ has a subgroup $\langle h,k\rangle$ isomorphic to the modular group $\Z_2*\Z_3$, where $h$ is an order-two quasisymmetry that switches the two components of $J_f\setminus \{p\}$ (mapping the top component to the bottom by $h$, and the bottom component to the top by the inverse of this), and $k$ is rotation by $2\pi/3$ at~$0$. The authors are not aware of any other quasisymmetries of this Julia set, and we conjecture that the full quasisymmetry group of $J_f$ is isomorphic to the modular group. 
\end{example}

\begin{remark}
The quasisymmetries $h$ and $k$ defined in the proof of Theorem~\ref{thm:UnicriticalTheorem} can easily be extended to piecewise conformal maps on the filled Julia set~$K_f$, which defines an action of $\Z_n*\Z_d$ on~$K_f$.  If $A$ is the regulated arc from the critical point to the chosen rotational fixed point, then the union of the arcs in the orbit of $A$ under $\langle h,k\rangle$ is a tree $T$, and $\langle h,k\rangle$ acts on $T$ with two orbits of vertices (namely the orbits of the critical point and of $p$) and cyclic vertex stabilizers.  In particular, $T$ is a geometric realization of the Bass--Serre tree for $\Z_n * \Z_d$.  For $n\geq 3$, we expect that the quasisymmetry group contains the uncountable subgroup of the automorphism group of $T$ consisting of automorphisms that preserve the counterclockwise order of edges at all vertices of degree~$d$.
\end{remark}

\subsection{Quasisymmetry groups that contain \texorpdfstring{$\boldsymbol{F}$}{F}}
\label{subsec:quasisymmetrygroupscontainingF}

The goal of this section is to prove the following theorem, which is part (2) of Theorem~\ref{thm:BigTheorem} from the introduction.

\begin{theorem}\label{thm:ContainsF}
Let $f$ be a postcritically finite hyperbolic polynomial, and suppose one of the leaves of the Hubbard tree for $f$ is contained in a periodic cycle of local degree\/~$2$. Then the quasisymmetry group of $J_f$ contains Thompson's group~$F$.
\label{thm:rearrangementscontainF}
\end{theorem}

Here the \newword{Hubbard tree} $H_f$ for a postcritically finite polynomial $f$ is the union of the regulated arcs that join every pair of points in the postcritical set (not including~$\infty$).  Douady and Hubbard proved that $H_f$ is a finite topological tree, that every arc in $H_f$ is regulated, and that $f(H_f)\subseteq H_f$.  For example, Figure~\ref{fig:CubicJuliaHubbardTree} shows the Hubbard tree for a certain postcritically finite cubic polynomial. 
\begin{figure}
    \centering
\includegraphics[scale=1.25]{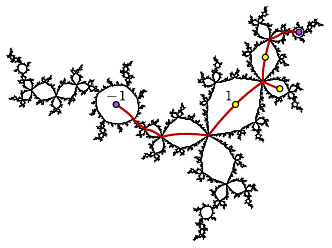}
\caption{The Julia set and Hubbard tree for the polynomial $f(z)=a\bigl(z^3-3z\bigr)+b$, where where $a\approx 0.0802+0.237i$ and $b\approx 1.9012+0.7383i$.  The critical point at $-1$ is part of the purple $2$-cycle, and the critical point at $1$ is part of the yellow $3$-cycle.}   \label{fig:CubicJuliaHubbardTree}
\end{figure}

The following lemma establishes some further properties of the projection functions $\pi_U$ that we will need for the proof.

\begin{lemma}\label{lem:ProjectionProperties}Suppose $f$ is a postcritically finite polynomial, and let $U$ be a bounded Fatou component for $f$.
\begin{enumerate}
    \item If $S\subseteq \partial U$, then the closure of $\pi_U^{-1}(S)$ is the union of $\pi_U^{-1}(S)$ with the closure of~$S$.\smallskip
    \item If $a\in\partial U$ and $\pi_U^{-1}(a)$ contains a critical point, then $\pi_{f(U)}^{-1}(f(a))$ contains a critical value. \smallskip
    \item If $a\in\partial U$ and $\pi_U^{-1}(a)$ contains no critical points, then $f$ maps $\pi_U^{-1}(a)$ homeomorphically to $\pi_{f(U)}^{-1}(f(a))$.
\end{enumerate}
\end{lemma}
\begin{proof}
For (1), let $\overline{S}$ be the closure of $S$, and let $X=\pi_U^{-1}(S)\cup \overline{S}$.  Clearly $X$ is contained in the closure of $\pi_U^{-1}(S)$, so it suffices to prove that $X$ is closed, i.e.~that
\[
K_f\setminus X =  \pi_U^{-1}\bigl(\overline{U}\setminus \overline{S}\bigr) \cup \bigcup_{a\in \overline{S}\setminus S} \bigl(\pi_U^{-1}(a)\setminus\{a\}\bigr)
\]
is relatively open in $K_f$.
Since $\overline{U}\setminus\overline{S}$ is relatively open in $\overline{U}$ and $\pi_U$ is continuous, the preimage $\pi_U^{-1}\bigl(\overline{U}\setminus\overline{S}\bigr)$ is relatively open in $K_f$, so it suffices to prove that $\pi_U^{-1}(a)\setminus\{a\}$ is relatively open in $K_f$ for each $a\in \partial U$.  But by statement~(i) above, $\pi_U^{-1}(a)\setminus \{a\}$ is a union of connected components of $K_f\setminus \overline{U}$.  Since $K_f$ is locally connected (by \cite[Theorem~18.3]{Milnor}) and $K_f\setminus \overline{U}$ is relatively open in $K_f$, each component of $K_f\setminus\overline{U}$ is relatively open in $K_f$.  We conclude that each $\pi_U^{-1}(a)\setminus\{a\}$ is relatively open in $K_f$, which proves that $X$ is closed.

For (2), suppose that $p$ is a critical point in $\pi_U^{-1}(a)$, and consider the regulated arc $[c,p]$. Replacing $p$ with another critical point if necessary, we may assume that $[c,p]$ has no critical points in its interior.  Then by \cite[Lemma~4.2]{OrsayNotes}, the function $f$ is one-to-one on $[c,p]$, and $f([c,p])$ is the regulated arc from $f(c)$ to~$f(p)$. This regulated arc goes through~$f(a)$, so $f(p)\in \pi_{f(U)}^{-1}(f(a))$, and hence $\pi_{f(U)}^{-1}(f(a))$ contains a critical value.

For (3), suppose $\pi_U^{-1}(a)$ contains no criticial points. We claim first that $f$ is one-to-one on~$\pi_U^{-1}(a)$.  To see this, let $p$ and $q$ be distinct points in $\pi_U^{-1}(a)$.  Since $[p,q]\subseteq [p,a]\cup [a,q]$, the entire regulated arc $[p,q]$ must lie in $\pi_U^{-1}(a)$.  Then $[p,q]$ does not contain any critical points of~$f$, so by \cite[Lemma~4.2]{OrsayNotes} the map $f$ is one-to-one on $[p,q]$.  In particular, $f(p)\ne f(q)$, which proves that $f$ is one-to-one on $\pi_U^{-1}(a)$. Since $\pi_U^{-1}(a)$ is compact and Hausdorff, it follows that $f$ maps $\pi_U^{-1}(a)$ homeomorphically to its image, so all that remains of (3) is to prove that $f(\pi_U^{-1}(a))=\pi_{f(U)}^{-1}(f(a))$.

To prove that  $f(\pi_U^{-1}(a))\subseteq \pi_{f(U)}^{-1}(f(a))$, let $p\in \pi_U^{-1}(a)$.  If $p=a$ then $f(p)=f(a)\in \pi_{f(U)}^{-1}(f(a))$ and we are done, so suppose $p\ne a$. Let $c$ be the center of~$U$, and observe that the regulated arc $[p,c]$ from $p$ to $c$ is the union of the interior ray $[a,c]$ in $\overline{U}$ with the regulated arc $[p,a]$.  Then $[p,c]$ has no critical points in its interior, so by \cite[Lemma~4.2]{OrsayNotes} the image $f([p,c])$ is the regulated arc from $f(p)$ to $f(c)$. This passes through~$f(a)$, which proves that $f(p)\in \pi_{f(U)}^{-1}(f(a))$, and therefore $f(\pi_U^{-1}(a))\subseteq \pi_{f(U)}^{-1}(f(a))$.

Finally, to prove that $f$ maps $\pi_U^{-1}(a)$ onto $\pi_{f(U)}^{-1}(f(a))$, let $p$ be a point in $\pi_{f(U)}^{-1}(f(a))$, and suppose first that $p$ does not lie in the postcritical set $P_f$. If $p=f(a)$ then we are done, so suppose $p\ne f(a)$.  Since each postcritical point lies in the interior of a Fatou component, we can find an arc $A$ in $\pi_{f(U)}^{-1}(f(a))$ that connects $f(a)$ to $p$ and avoids any postcritical points.  Applying the path-lifting property for the covering map $f\colon K_f\setminus f^{-1}(P_f)\to K_f\setminus P_f$, we can lift $A$ to an arc $\overline{A}$ in~$K_f\setminus f^{-1}(P_f)$ that connects $a$ to some point $\overline{p}\in f^{-1}(p)$.  Since $A\cap f(\overline{U})=\{f(a)\}$, we know that $\overline{A}\cap \overline{U}=\{a\}$, so $\overline{A}$ must lie in $\pi_U^{-1}(a)$.  Then $\overline{p}\in \pi_U^{-1}(a)$ and hence $p\in f(\pi_U^{-1}(a))$ in the case where $p$ is not postcritical.  Since $f(\pi_U^{-1}(a))$ is closed, any postcritical points in $\pi_{f(U)}^{-1}(f(a))$ lie in $f(\pi_U^{-1}(a))$ as well.
\end{proof}

\begin{proof}[Proof of Theorem~\ref{thm:ContainsF}] 
Let $f$ be a postcritically finite hyperbolic polynomial with Hubbard tree $H_f$, and let $c$ be a leaf of $H_f$ which is contained in a periodic cycle of local degree two. Replacing $f$ by an iterate of $f$ (an operation which does not change the Hubbard tree), we may assume that $c$ is a fixed critical point of local degree~$2$.  We will use the ping-pong lemma for Thompson's group~$F$ (see Appendix~\ref{ap:PingPongLemmas}) to produce a subgroup of the quasisymmetry group isomorphic to~$F$.

Since $c$ is a super-attracting fixed point, it lies in some Fatou component $U$ for~$f$.  By \cite[Theorem~133.3]{Steinmetz}, we know that $U$ is a Jordan domain, so by B\"ottcher's and Carath\'{e}odory's theorems there exists a homeomorphism  $\Phi\colon \overline{\D}\to \overline{U}$ that satisfies $\Phi(0)=c$ and $\Phi(z^2)=f(\Phi(z))$ for all $z\in\overline{U}$.  Let $\gamma\colon \R\to \partial U$ be the Carath\'eodory loop, i.e.\ the map $\gamma(\theta)=\Phi(e^{i\theta})$, and note that $f(\gamma(t))=\gamma(2t)$ for all $t\in \R$.
Since $c$ is a leaf of $H_f$, the intersection $H_f\cap U$ must be a single internal ray of~$U$.  Since $f(H_f)\subseteq H_f$, this must be precisely the internal ray connecting $c$ and $\gamma(0)$.

For $S\subseteq \R$, let $L_S$ denote $\pi_U^{-1}(\gamma(S))$, and let $JL_S=J_f\cap L_S$. Since all of the postcritical points of $f$ lie on the Hubbard tree, all of the postcritical points except $c$ must lie in $L_{\{0\}}$. By Lemma~\ref{lem:ProjectionProperties}(2), it follows that all critical points of $f$ other than $c$ are contained in $L_{\{0\}}\cup L_{\{1/2\}}$. 
 By Lemma~\ref{lem:ProjectionProperties}(3), we conclude that $f$ maps $L_{\{t\}}$ homeomorphically to $L_{\{2t\}}$ for all $t\in (0,1/2)\cup (1/2,1)$.

Define a function $g_0\colon J_f\to J_f$ as follows:
\begin{enumerate}
\item $g_0$ is the identity on $JL_{\{0\}}$.\smallskip
\item $g_0$ maps $JL_{(0,1/4]}$ homeomorphically to $JL_{(0,1/2]}$ via~$f$.\smallskip
\item $g_0$ maps $JL_{(1/4,1/2)}$ homeomorphically to $JL_{(1/2,3/4)}$ via $h_2^{-1}\circ h_1$, where
\[
h_1\colon JL_{(1/4,1/2)}\to JL_{(0,1)}
\qquad\text{and}\qquad
h_2\colon JL_{(1/2,3/4)}\to JL_{(0,1)}
\]
are restrictions of $f^2$.\smallskip
\item $g_0$ maps $JL_{[1/2,1)}$ homeomorphically to $JL_{[3/4,1)}$ by a branch of~$f^{-1}$.\smallskip
\end{enumerate}
\begin{figure}[tb]
    \centering
    $\raisebox{-0.47\height}{\fbox{\includegraphics{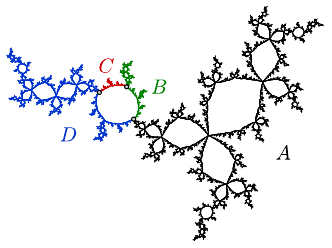}}}\hfill\longrightarrow\hfill \raisebox{-0.47\height}{\fbox{\includegraphics{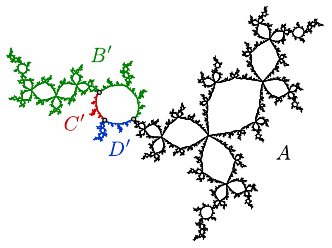}}}$
    \caption{The homeomorphism $g_0$ for the cubic Julia set from Figure~\ref{fig:CubicJuliaHubbardTree} with critical point~$-1$, where $A=JL_{\{0\}}$, $B=JL_{(0,1/4]}$, $B'=JL_{(0,1/2]}$, $C=JL_{(1/4,1/2)}$, $C'=JL_{(1/2,3/4)}$,  $D=JL_{[1/2,1)}$, and $D'=JL_{[3/4,1)}$.  Here $f$ is the second iterate of the cubic polynomial, under which $-1$ is fixed.}
    \label{fig:CubicJuliaRearrangement}
\end{figure}
See Figure~\ref{fig:CubicJuliaRearrangement} for an example of such a~$g_0$. Note that $JL_{\{0\}}$ is closed, and by Lemma~\ref{lem:ProjectionProperties}(1), we have
\[ \overline{JL_{(0,1/4]}}= JL_{(0,1/4]}\cup\{\gamma(0)\},\qquad 
\overline{JL_{(1/4,1/2)}}= JL_{(1/4,1/2)}\cup\{\gamma(1/4),\gamma(1/2)\},
\]
and
\[
\overline{JL_{[1/2,1)}}= JL_{[1/2,1)}\cup\{\gamma(0)\}.
\]
In particular, the four pieces of $g_0$ agree on the intersections of their closures, with
\[
g_0(\gamma(0))=\gamma(0),\qquad
g_0(\gamma(1/4))=\gamma(1/2),\qquad\text{and}\qquad
g_0(\gamma(1/2)) = \gamma(3/4),
\]
which proves that $g_0$ is continuous.  Since $g_0$ is bijective and $J_f$ is compact, we conclude that $g_0$ is a homeomorphism.  Furthermore, since the points $\gamma(0)$, $\gamma(1/4)$, $\gamma(1/2)$, and $\gamma(3/4)$ are periodic or pre-periodic and each of the four pieces of $f$ is canonical, it follows from Theorem~\ref{thm:PiecewiseCanonicalQuasisymmetries} that $g_0$ is a quasisymmetry.

Similarly, let $g_1\colon J_f\to J_f$ be the function defined as follows:
\begin{enumerate}
\item $g_1$ is the identity on $JL_{[0,1/2]}$.\smallskip
\item $g_1$ maps $JL_{(1/2,5/8]}$ homeomorphically to $JL_{(1/2,3/4]}$ via~$k_2^{-1}\circ k_1$, where
\[
k_1\colon JL_{(1/2,5/8]}\to JL_{(0,1/2]}
\qquad\text{and}\qquad
k_2\colon JL_{(1/2,3/4]}\to JL_{(0,1/2]}
\]
are restrictions of $f^2$ and $f$, respectively.\smallskip
\item $g_1$ maps $JL_{(5/8,3/4)}$ homeomorphically to $JL_{(3/4,7/8)}$ via $\ell_2^{-1}\circ \ell_1$, where
\[
\ell_1\colon JL_{(5/8,3/4)}\to JL_{(0,1)}
\qquad\text{and}\qquad
\ell_2\colon JL_{(3/4,7/8)}\to JL_{(0,1)}
\]
are restrictions of $f^3$.\smallskip
\item $g_1$ maps $JL_{[3/4,1)}$ homeomorphically to $JL_{[7/8,1)}$ by a branch of~$f^{-1}$.\smallskip
\end{enumerate}
Again, note that these four pieces agree on the intersections of their closures, so $g_1$ is a homeomorphism and indeed a quasisymmetry.

To apply the ping-pong lemma for Thompson's group $F$ (see Appendix~\ref{ap:PingPongLemmas}), let $R$ be the set $JL_{(1/2,1)}$.  Then:
\begin{enumerate}
    \item $g_0(R)=JL_{(3/4,1)}$ is a proper subset of $R$,\smallskip
    \item $g_1$ is the identity on $J_f-R = JL_{[0,1/2]}$, and\smallskip
    \item $g_1$ agrees with $g_0$ on $g_0(R)=JL_{(3/4,1)}$,
\end{enumerate}
so we conclude that $\langle g_0,g_1\rangle$ is isomorphic to~$F$.\end{proof}

\begin{remark}
The \newword{extended Hubbard tree} $\widehat{H}_f$ of $f$ is the union of the regulated arcs that join every pair of points in $P_f\cup C_f$, where $P_f$ is the postcritical set for $f$ and $C_f$ is the set of critical points.  If $c$ is a fixed critical point of $f$ of local degree $2$ which is a leaf of $\widehat{H}_f$, then $L_{\{1/2\}}$ contains no critical points, and therefore maps homeomorphically to $L_{\{0\}}$ under $f$.  In this case, we can use the method of proof of Theorem~\ref{thm:ContainsF} to obtain a subgroup of the quasisymmetry group isomorphic to Thompson's group~$T$.  For example, applying this when $f$ is the second iterate of $z^2-1$ and $c=-1$, we obtain one of the copies of $T$ in the basilica Thompson group (see~\cite{BeFo1}). \end{remark}

\begin{remark}
If the Hubbard tree for a postcritically finite hyperbolic polynomial $f$ has a fixed critical point $c$ whose local degree is some integer $n\geq 3$, one might expect that Theorem~\ref{thm:ContainsF} would generalize to give a copy of the $n$-ary Thompson group $F_n$ inside the quasisymmetry group.  However, it turns out that the proof of Theorem~\ref{thm:ContainsF} does not go through for $n\geq 3$.

For example, suppose $c$ is a leaf of $H_f$ which is a fixed critical point of local degree~$3$. Let $U$ be the Fatou component containing $c$, choose a B\"ottcher homeomorphism $\Phi\colon \overline{U}\to\overline{\D}$ satisfying $\Phi(c)=0$ and $\Phi(f(z))=\Phi(z)^3$, and define $L_S$ as in the proof of Theorem~\ref{thm:ContainsF}.  Then both $L_{\{1/3\}}$ and $L_{\{2/3\}}$ might contain critical points, so there need not be any canonical homeomorphism from either of these to $L_{\{0\}}$, and indeed there may be no canonical homeomorphism $L_{\{1/3\}}\to L_{\{2/3\}}$.
More generally, if $p=a/3^j$ and $q=b/3^k$ where $a\equiv 1\;(\mathrm{mod}\;3)$ and $b\equiv 2\;(\mathrm{mod}\;3)$, there is not necessarily any canonical homeomorphism $L_{\{a\}}\to L_{\{b\}}$.
This difficulty leaves us unable to construct any nontrivial elements of the $3$-ary Thompson group~$F_3$.  For example, the piecewise linear homeomorphism
\[
g(t) = \begin{cases}3t & \text{if }t\in \bigl[0,\tfrac19\bigr], 
 \\[3pt] t+\tfrac29 & \text{if }t\in \bigl[\tfrac19,\tfrac23\bigr], \\[3pt]
 \tfrac13 t + \tfrac23 & \text{if }t\in \bigl[\tfrac23,1\bigr],
 \end{cases}
\]
lies in $F_3$ and maps $1/3$ to $5/9$, but there need not be a canonical homeomorphism $L_{\{1/3\}}\to L_{\{5/9\}}$, so we cannot necessarily find a piecewise canonical homeomorphism that acts as $g$ on $\partial U$.  \end{remark}

\appendix
\section{Removing Finitely Many Points}
\label{ap:RemovingFiniteSets} 

The goal of this section is to prove the following theorem.

\begin{theorem}\label{thm:FinitelyManyComponents}
Let $f\colon \Chat\to \Chat$ be a hyperbolic rational function with connected Julia set~$J_f$, and let $S$ be a finite set of periodic or preperiodic points in~$J_f$.  Then $J_f\setminus S$ has finitely many components.
\end{theorem}

We need the following result about the topology of~$J_f$.

\begin{lemma}\label{lem:ComplementaryComponents}Let $f\colon\Chat\to\Chat$ be a hyperbolic rational map with connected Julia set~$J_f$.
\begin{enumerate}
    \item If $p\in J_f$ is a periodic point, then any Fatou component that has $p$ on its boundary is periodic.\smallskip
    \item If $p\in J_f$ is periodic or preperiodic, then there exist at most finitely many Fatou components $U$ for which $p\in \partial U$, and for any such $U$ the complement $\partial U\setminus \{p\}$ has finitely many connected components.
\end{enumerate}
 \end{lemma}
\begin{proof}
Suppose first that $p$ is periodic, and let $U$ be a Fatou component that has $p$ on its boundary.  Replacing $f$ by an iterate of $f$, we may assume that $p$ is a fixed point. Since $J_f$ is connected, we know that $U$ is simply connected.  Since $f$ is hyperbolic, it follows that $\partial U$ is locally connected~\cite[Lemma~19.3]{Milnor}, so $p$ is accessible from~$U$.  We wish to prove that $U$ is periodic.

By Sullivan's nonwandering theorem (see~\cite[Theorem~16,4]{Milnor}), we know that $U$ is periodic or preperiodic.  Replacing $f$ by an iterate, there exists a $k\geq 0$ so that $V = f^k(U)$ is a fixed Fatou component.  Again $\partial V$ is locally connected, so by Carath\'{e}odory's theorem there exists a map $\Phi\colon \overline{\D}\to \overline{V}$ that restricts to a conformal isomorphism $\D\to V$.  Note that there exists an analytic map $g\colon \overline{\mathbb{D}}\to\overline{\mathbb{D}}$ (namely a finite Blaschke product) such that $\Phi\circ g = f\circ \Phi$.  By a theorem of Pommerenke~\cite[Theorem~2]{Pommer}, there exists a $\theta\in \partial\mathbb{D}$ which is periodic under $g$ such that~$\Phi(\theta)=p$.

This value of $\theta$ determines a periodic access to $p$ in~$V$.  By a theorem of Petersen \cite[Theorem~A]{Petersen}, it follows that any access to $p$ from $V$ is periodic.  Now, since $f$ is one-to-one in a neighborhood of~$p$, the induced action of $f$ on accesses to $p$ from the Fatou set is one-to-one.  But $f^k$ maps some access to $p$ from $U$ to a periodic access to $p$ from $V$, so it follows that $U=V$, which proves~(1).

For (2), in the case where $p$ is periodic, Sullivan proved that a rational map has only finitely many periodic Fatou components~\cite{Sullivan}.  By (1), it follows that $p$ lies on the boundary of only finitely many Fatou components.  Moreover, if $U$ is any such component and $\Phi\colon \overline{\mathbb{D}}\to \overline{U}$ is a map as above, then $\Phi^{-1}(p)$ must be finite since $U$ has only finitely many accesses to~$p$.  In particular, $\partial\mathbb{D}\setminus \Phi^{-1}(p)$ has only finitely many connected components, and therefore $\partial U\setminus \{p\}$ has only finitely many connected components as well.

All that remains is the case where $p$ is preperiodic.  In this case, there exists an $n\geq 1$ so that $f^n(p)$ is a periodic point.  By the above argument, $f^n(p)$ lies on the boundary of only finitely many Fatou components, and since each Fatou component has finitely many preimages under $f^n$ it follows that $p$ lies on the boundary of only finitely many Fatou components.  Moreover, if $U$ is a Fatou component that has $p$ on its boundary, then the mapping $f^n\colon U\to f^n(U)$ induces a finite-to-one mapping from the accesses to $p$ in $U$ to the accesses to $f^n(p)$ in~$f^n(U)$.  Since there are only finitely many accesses to $f^n(p)$ in $f^n(U)$ by the argument above, it follows that there are only finitely many accesses to $p$ in~$U$, and therefore $\partial U\setminus \{p\}$ has finitely many connected components.
\end{proof}

%If $J$ is a locally connected closed subset of $\Chat$ and $E\subseteq J$, define the \newword{filling} of $E$ to be the union of $E$ with all the components of $\Chat\setminus J$ whose boundaries are contained in $E$.  If $S$ is a closed subset of $J$, it is easy to prove that the filling of $J\setminus S$ has the same number of connected components as $J\setminus S$.

If $J_f$ is a connected Julia set and $E\subseteq J_f$, define the \newword{filling} of $E$ to be the union of $E$ with all the Fatou components whose boundaries are contained in~$E$.

\begin{lemma}\label{lem:ComponentsFilling}
Let $f\colon \Chat\to \Chat$ be a hyperbolic rational map with connected Julia set~$J_f$, and let $S$ be a closed subset of $J_f$. Then the filling of $J_f\setminus S$ has the same number of connected components as $J_f\setminus S$.
\end{lemma}
\begin{proof}
Let $F$ be the filling of $J_f\setminus S$, let $\{E_\alpha\}_{\alpha\in\mathcal{I}}$ be the connected components of $J_f\setminus S$, and let $F_\alpha$ denote the filling of $E_\alpha$.  We claim that the $F_\alpha$ are precisely connected components of~$F$.

Clearly each $F_\alpha$ is connected.  Furthermore, since $J_f$ is connected, every Fatou component for $f$ has connected boundary, so any Fatou component which is contained in $F$ must be contained in one of the $F_\alpha$.  It follows that $F$ is the disjoint union of the $F_\alpha$'s, so it suffices to prove that each $F_\alpha$ is open in $F$.

Let $q\in F_\alpha$.  If $q$ lies in some Fatou component, then that Fatou component is a neighborhood of $q$ in $F_\alpha$. Suppose then that $q\in E_\alpha$.  Since $J_f$ is locally connected and $S$ is closed, there exists a neighborhood $V$ of $q$ in $\Chat$ which is disjoint from $S$ such that $V\cap J_f$ is connected.  In particular, $V$ does not intersect any other $E_\beta$. Then $V$ does not intersect any other $F_\beta$, for if $U$ is a Fatou component contained in $F_\beta$, then $V$ does not intersect $\partial U$ and hence does not intersect~$U$. We conclude that $V\cap F\subseteq F_\alpha$, which proves that $F_\alpha$ is open in~$F$.
\end{proof}

\begin{proof}[Proof of Theorem~\ref{thm:FinitelyManyComponents}]
Let $f\colon \Chat\to \Chat$ be a hyperbolic rational function with connected Julia set~$J_f$, and let $S$ be a finite set of periodic or preperiodic points in~$J_f$.  We wish to prove that $J_f\setminus S$ has finitely many components.  By Lemma~\ref{lem:ComponentsFilling}, it suffices to prove that the filling of $J_f\setminus S$ has only finitely many connected components.

Since each point of $S$ is periodic or preperiodic, it follows from  Lemma~\ref{lem:ComplementaryComponents} that there exist only finitely many Fatou components $U_1,\ldots, U_n$ whose boundaries intersect~$S$.  Let $U=U_1\cup\cdots \cup U_n$, and observe that the filling of $J_f\setminus S$ is precisely the complement of $U\cup S$.  Thus it suffices to prove that $\Chat\setminus (U\cup S)$ has only finitely many connected components.

Let $C$ be a connected component of $\Chat\setminus (U\cup S)$.  Then $\partial C\subseteq \partial U\cup S$, and since $\partial C$ is infinite it follows that $\partial C$ intersects $\partial U\setminus S$.  Since $\partial U \setminus S \subseteq \Chat\setminus (U\cup S)$, we conclude that $C$ contains a connected component of $\partial U\setminus S$.  But $\partial U\setminus S$ has finitely many connected components by Lemma~\ref{lem:ComplementaryComponents}, and therefore $\Chat\setminus (U\cup S)$ has only finitely many connected components as well.
\end{proof}

\section{Ping-Pong Lemmas}
\label{ap:PingPongLemmas} 

In this section we state ping-pong lemmas for free products and for Thompson's group~$F$.  The version for free products is well-known.

\begin{pingpongfree}
Let $G$ be a group acting on a set $X$, and let $H$ and $K$ be subgroups of $G$. Suppose there exist subsets $X_H,X_K\subset X$ such that
\begin{enumerate}
    \item $h(X_K)$ is a proper subset of $X_H$ for all nontrivial $h\in H$, and\smallskip
    \item $k(X_H)$ is a proper subset of $X_K$ for all nontrivial $k\in K$.
\end{enumerate}
Then the subgroup of $G$ generated by $H$ and $K$ is isomorphic to $H*K$.
\end{pingpongfree}
\begin{proof}
See \cite[II.24]{DLH}. \end{proof}

We also need a ping-pong type lemma for Thompson's group~$F$.  The following lemma is related to, but not the same as, the ping-pong lemma for $F$ proven by Bleak, Brin, Kassabov, Moore, and Zaremsky in~\cite{BBKMZ}.

\begin{pingpongF}
Let $G=\langle g_0,g_1\rangle$ be a group acting faithfully on a set $X$, and suppose there exists a set $R\subset X$ such that
\begin{enumerate}
    \item $g_0(R)$ is a proper subset of $R$, \smallskip
    \item $g_1$ is the identity on $X-R$, and\smallskip
    \item $g_1$ agrees with $g_0$ on $g_0(R)$.
\end{enumerate}
Then $G$ is isomorphic to Thompson's group~$F$.
\end{pingpongF}

Before we prove this lemma, recall that Thompson's group $F$ is the group of homeomorphisms of $[0,1]$ generated by the functions
\[\
x_0(t) = \begin{cases}2t & \text{if }t\in\bigl[0,\frac{1}{4}\bigr], \\[3pt]
t+\frac{1}{4} & \text{if } t\in \bigl[\frac{1}{4},\frac{1}{2}\bigr], \\[3pt]
\frac{1}{2}t+\frac{1}{2} & \text{if }t\in \bigl[\frac{1}{2},1\bigr],\end{cases}
\qquad\text{and}\qquad
x_1(t) = \begin{cases}t & \text{if }t\in\bigl[0,\frac{1}{2}\bigr], \\[3pt] 2t-\frac{1}{2} & \text{if }t\in\bigl[\frac{1}{2},\frac{5}{8}\bigr], \\[3pt]
t+\frac{1}{8} & \text{if } t\in \bigl[\frac{5}{8},\frac{3}{4}\bigr], \\[3pt]
\frac{1}{2}t+\frac{1}{2} & \text{if }t\in \bigl[\frac{3}{4},1\bigr].\end{cases}
\]
Note that $x_0$ and $x_1$ satisfy the hypotheses of the proposition with $X=[0,1]$ and $R=\bigl[\frac{1}{2},1\bigr]$, where $x_0(R) = \bigl[\frac{3}{4},1\bigr]$.

With respect to these generators, $F$ has presentation
\[
\bigl\langle x_0,x_1 \;\bigr|\; x_0^2 x_1x_0^{-2} = (x_1x_0)x_1(x_1x_0)^{-1},x_0^3x_1x_0^{-3}=(x_1x_0^2)x_1(x_1x_0^2)^{-1}\bigr\rangle.
\]
See \cite[Theorem 3.1]{CFP} for a proof, where these functions $x_0$ and $x_1$ are denoted $A^{-1}$ and $B^{-1}$.

\begin{proof}
Since $g_0$ and $g_1$ agree on $g_0(R)$, we know that $g_0^2$ and $g_1g_0$ agree on~$R$. Similarly, since $g_0$ and $g_1$ agree on $g_0^2(R)\subseteq g_0(R)$, we know that $g_0^3$ and $g_1g_0^2$ agree on~$R$.  Since $g_1$ is the identity on $X-R$, it follows that
\[
g_0^2 g_1 g_0^{-2} = (g_1g_0)g_1(g_1g_0)^{-1} \qquad\text{and}\qquad g_0^3 g_1 g_0^{-3}
 = (g_1g_0^2)g_1(g_1g_0^2)^{-1}.
 \]
Thus there exists an epimorphism $\varphi\colon F\to G$ defined by $\varphi(x_0)=g_0$ and $\varphi(x_1)=g_1$.

To prove that $\varphi$ is an isomorphism, recall that every proper quotient of $F$ is abelian \cite[Theorem~4.3]{CFP}.  Thus, it suffices to prove that $G$ is nonabelian, i.e. that $g_0$ and $g_1$ do not commute. Since $g_1$ is the identity on $X-R$, we know that $g_1(R)=R$, so $g_1$ does not agree with $g_0$ on $R$.  Then $g_0g_1g_0^{-1}$ does not agree with $g_0g_0g_0^{-1}=g_0$ on $g_0(R)$.  Since $g_1$ agrees with $g_0$ on~$g_0(R)$, we conclude that $g_0g_1g_0^{-1}\ne g_1$, so $g_0$ and $g_1$ do not commute.
\end{proof}

\subsection*{Acknowledgements}
We would like to thank Mikhail Lyubich, Sergiy Merenkov, and Matteo Tarocchi for helpful conversations and suggestions. We would also like to thank Collin Bleak, Peter Cameron, Justin Moore, the University of St Andrews, and Stockton University for providing funds for travel and lodging in support of this work.  Finally, we would like to thank several undergraduate students at Bard College and Stockton University for helpful conversations on questions related to this research, including Rebecca Claxton, Wayne Laffitte, Yuan Liu, Emily Mahler, Will Smith, Jasper Weinburd, and Shuyi Weng.

\bigskip
\newcommand{\doi}[1]{\href{https://doi.org/#1}{Crossref}}
\newcommand{\arXiv}[1]{\href{https://arxiv.org/abs/#1}{arXiv}}
\newcommand{\biblink}[2]{\href{#1}{#2}}
\renewcommand{\MR}[1]{\href{https://mathscinet.ams.org/mathscinet-getitem?mr=#1}{\mbox{MathSciNet}}}
\bibliographystyle{plain}

\end{document}